\documentclass[reqno,12pt,letterpaper]{amsart}
\usepackage[proof]{newmzmacros2}
\usepackage[all,cmtip]{xy}

\newcommand{\RR}{{\mathbb R}}

\usepackage{dsfont}
\usepackage{soul}
\usepackage{xcolor}
\usepackage{fancyvrb}
\usepackage{subcaption}
\usepackage{algorithm}
\usepackage{algorithmicx}
\usepackage{algpseudocode}

\def\blue#1{\textcolor{blue}{#1}}

\title{1D Scattering through time dependent media with memory}

\author{Jeffrey Galkowski}
\email{j.galkowski@ucl.ac.uk}
\address{Department of Mathematics, University College London, WC1H 0AY, UK}

\author{Maciej Zworski} 
\email{zworski@math.berkeley.edu}
\email{hertz@math.berkeley.edu}
\address{Department of Mathematics, University of California, Berkeley, CA 94720}

\begin{document}

\maketitle

\vspace{-0.2in}

\begin{center}
{\sc With an appendix by Zhen Huang and Maciej Zworski}
\end{center}

\vspace{-0.1in}

\begin{abstract}
We construct a scattering matrix with operator valued entries describing solutions to the 1+1 wave equation where
permittivities has memory and depends on time and space. It is the analogue of the scattering matrix 
for spatially localised perturbations where the entries are functions of frequency and appear as Fourier multipliers in solutions of the wave equation. This provides a mathematical explanation of the numerical construction in the recent paper by Horsley et al \cite{horse}. 
\end{abstract}

\section{Introduction}
\label{s:intr}

There is a considerable interest in materials whose properties depend on time and which have memory. 
Having memory essentially means having a (causal) dependence on frequency -- see for instance \cite{galif1} for a survey and 
\cite{galif2} for a recent experimental and theoretical study. Our motivation comes from
a letter \cite{horse} by Horsley, Galiffi,  and Wang and we refer to it for more references to the literature and the physics background.

\begin{figure}
\includegraphics[width=14cm]{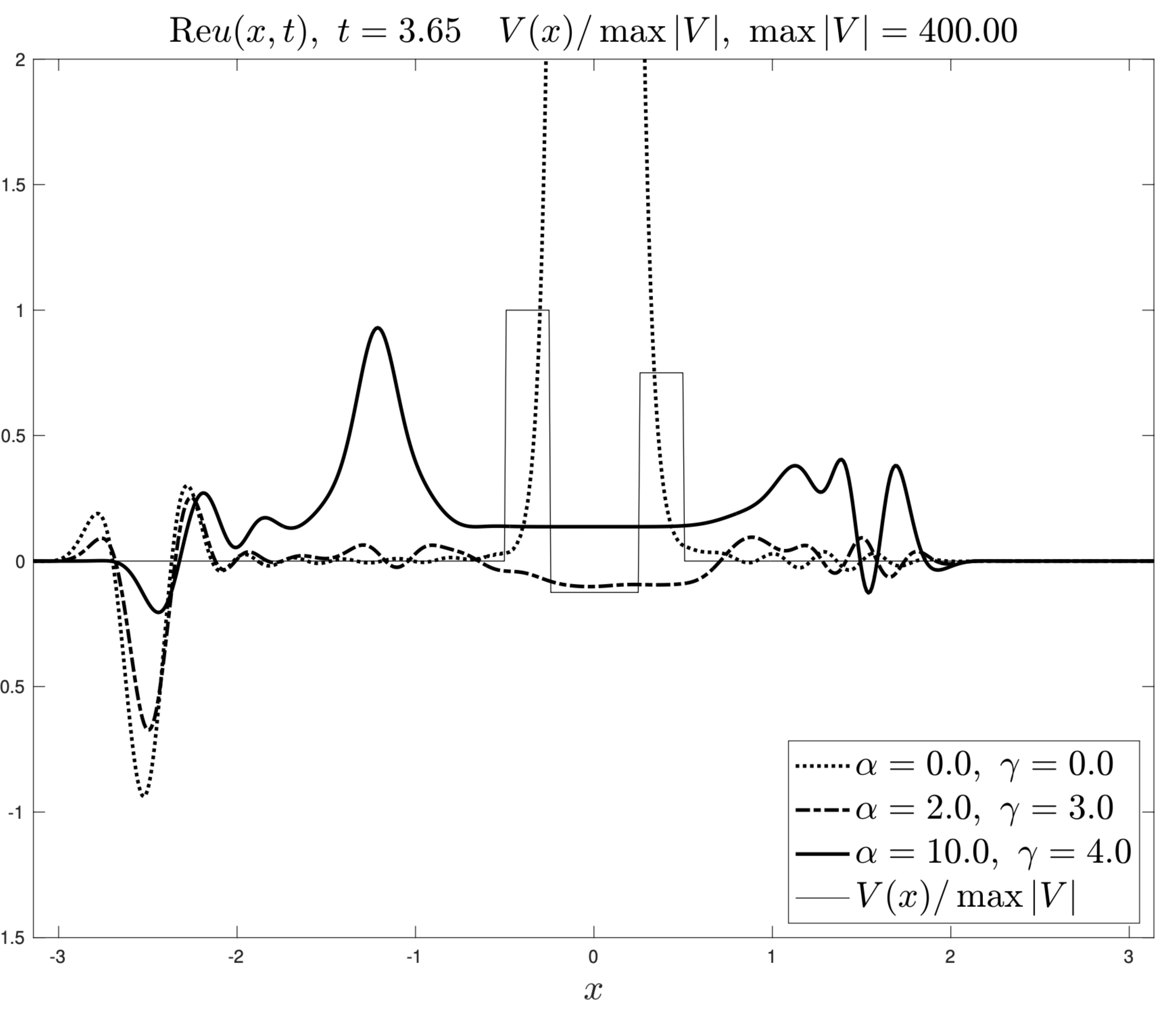}
\caption{\label{f:1} Comparisons of evolutions of a gaussian packet $ g ( t - x  ) $, $ t \ll 1 $, 
$ g (y ) = e^{-(x_0+y )^2) / \sigma - i \lambda  (y- x_0))} $ , $ \sigma = 0.05$, $ \lambda = 10$, $x_0 = - 2$ for different
values of the parameters in \eqref{eq:Bex} (with $ m = 1 $) -- this corresponds to the model considered in \cite{horse}.
An animated version is available at \url{https://math.berkeley.edu/~zworski/wave_multi_one.mp4}. The code for producing this movie and the figure is enclosed in the appendix.}
\end{figure}

\subsection{Scattering for permittivities with memory}

\begin{equation}
\label{eq:newave}
\begin{gathered}   D_t^2 u ( t, x ) - a ( x, t ) \int_{-\infty}^t e^{ - \gamma ( t - t' )}  D_t u ( t' , x ) dt' - D_x^2 u ( t, x ) = 0 ,
, \\
 a \in L^\infty ( \mathbb R_x ; C_{\rm{c}}^\infty ( \mathbb R_t ) )  , \ \  \supp a  \subset [-R, R ] \times [ -T , T ] , \ \ \  
 \gamma  > 0 , \\
u ( t , x)|_{ t \leq 0 }  = g ( x - t ) , \ \ \  \supp g \subset ( - \infty , - R ) . 
 \end{gathered}
 \end{equation}
 (The compact support in time could be relaxed to superexponential decay but we restrict ourselves to the simplest case here).
 
 At least formally this corresponds on the Fourier transform side to 
\begin{equation}
\label{eq:statwave}   
\begin{gathered}
P := D_x^2 - \omega^2 + A ( x ) , \ \\
A( x ) := a ( x, D_\omega ) \frac{ \omega }{ \omega + i \gamma } :
 H_r ( \mathbb R_\omega ) \to H_r ( \mathbb R_\omega ) ,
 \end{gathered}\end{equation}
 where
\begin{equation}
\label{eq:defHca}
H_{\alpha} := \left\{ f : 
f (  \bullet ) \in \mathscr O ( \mathbb C_+ ) , \ \ 
\sup_{\sigma > 0 }  e^{ - 2\sigma \alpha } \int_{\mathbb R } |f(\lambda + i\sigma)|^2 d\lambda  < + \infty \right \} ,
\end{equation}
and also write $H_\infty:=\cup_ \alpha H_{\alpha}.$ The Hardy spaces $ H_{\alpha} $ appear naturally 
in scattering theory for the wave equation since the work of Lax--Phillips \cite{LP} -- see \S \ref{s:revs}, and 
\eqref{eq:mapping1} specifically, for the review of the simplest case. 

The action of $a(x,D_\omega)$ is given by
\begin{equation}
\label{e:aAction}
[a(x,D_\omega)u](x,\omega+i\sigma):=  \frac{1}{2\pi} \int  \widehat{a(x,\bullet)} (\omega-\lambda)u(x,\lambda+i\sigma)d\lambda.
\end{equation}

We introduce the following Hilbert space of functions of position and frequency, $ f : \mathbb R \times \mathbb C_+ \to \mathbb C$,  $\mathbb C_+ := \{ \omega \in \mathbb C: \Im \omega > 0 \} $, 
\begin{equation}
\label{eq:defHca}
\mathscr H^{s}_{\alpha}:= \left\{ f : 
f ( x , \bullet ) \in \mathscr O ( \mathbb C_+ ) , \ \ 
\sup_{\sigma > 0 }  e^{ - 2\sigma \alpha } \int_{\mathbb R } \|f(\cdot,\lambda + i\sigma)\|_{H^s(\mathbb{R})}^2 d\lambda  < + \infty \right \} . 
\end{equation}
We write $\mathscr{H}_\alpha :=\mathscr{H}_\alpha ^0$, $\mathscr{H}_\infty=\cup_{\alpha}\mathscr{H}_\alpha$, and
$$
\mathscr{H}_{\alpha,\loc}^s:\{ u\,:\, u(x,\bullet)\in \mathscr{O}(\mathbb{C}_+)\text{ and } \chi(x) u(x,\omega)\in \mathscr{H}_{\alpha}^s\,\text{ for all }\chi\in C_c^\infty\}.
$$

The goal now is to prove the following analogue of the existence of the scattering matrix \eqref{eq:defS}: 
\begin{theo}
\label{t:1}
Suppose that assumptions in \eqref{eq:newave} hold and $ f \in  ( \omega + i )^{-\blue{1}}  H_{-R } $. Then there exist bounded operators
\[   T , R_+ :  ( \omega + i )^{-\blue{1}}  H_{\alpha} \to    \omega^{-1} H_{\alpha + 2 R}  , \ \ \ \alpha \in \mathbb R , \]
such that, for $ P $ given in \eqref{eq:statwave}, there exists a unique solution, 
$ u ( x, \omega )\in \omega^{-1}\mathscr{H}_{\infty,\loc}$, $ ( x, \omega ) \in \mathbb R^2 $,   of $ P u = 0 $ such that 
\begin{equation}
\label{eq:scat2}
u ( x, \omega ) = \left\{ \begin{array}{ll}  f ( \omega ) e^{ i \omega x } + R_+ f ( \omega ) e^{ - i \omega x }, 
& x < -R , \\ 
 T f   ( \omega ) e^{ i \omega x }  , & x > R . \end{array} \right. 
\end{equation}
\end{theo}
The operators $ R_+ $ and $ T $ were constructed numerically in \cite{horse} for the case when $ a ( x, t ) = 
V_0 \indic_{ -R, R } ( x ) \chi ( t ) $. That was done by following the construction of the scattering matrix for 
step potentials (see \S \ref{s:revs} and \cite[Exercise 2.10.6]{res}) but with the exponentials 
\[ \text{ $  e^{ \pm \omega x } $  for $ |x| > R \ \ \ $ and
$ \ \  e^{\pm i \sqrt{ \omega^2 - V_0 } x } $ for $ |x| < R $,}\]
 replaced Schr\"odinger propagators  
$ e^{ \pm i x \chi( D_\omega ) \omega/( \omega + i \gamma ) } $ for $ |x| < R$. The needed inversion of operators 
was established numerically. The point of Theorem \ref{t:1} is that, under the assumption of localisation in 
space and time, the operator exists even for a larger class of perturbations.

The proof has two parts: the first, in \S \ref{s:constr}, is a functional analytic setup for constructing $ u ( x, \omega ) $. The second one, in \S \ref{s:none}, 
is the proof of non-existence of purely outgoing solutions to $ P u = 0 $, that is solutions satisfying 
\eqref{eq:scat2} with $ f \equiv 0 $ but the terms corresponding to $ R_+ f $ and $ T f $ potentially nonzero. 

\subsection{The wave equation}

We now describe how the operators constructed in Theorem \ref{t:1} appear in the 
wave evolution.  The theorem in the exact analogue of Proposition \ref{p:1Dwave} 
in which the standard scattering matrix for compactly supported 1D potentials appears in the description of 
scattered waves. 
\begin{theo}
\label{t:2}
Suppose that $ g \in C_{\rm{c}}^\infty ( (  R, \infty ) ) $ and that $ u ( t, x ) $ is the unique solution of
\eqref{eq:newave} satisfying $ u ( t , x) = g ( t-x) $, $ t < 0 $. Then 
\begin{equation}
\label{eq:scat2wave2}
u ( t , x ) = \left\{ \begin{array}{ll}  g (   t-x ) + \mathscr R_+   g  ( x+t  ) , & x < - R , \\
\mathscr T g ( t - x ) , & x > R , \end{array} \right. \end{equation}
where 
\[    \widehat { \mathscr R_+ g } ( \omega ) :=  R_+ \widehat g ( \omega ) ,  \ \ \ \  \widehat { \mathscr T  g } ( \omega ) :=  T \widehat g ( \omega ).
\]
and $ T $ and $ R_+ $ are described in Theorem \ref{t:1}.
\end{theo}
Existence of operators $ \mathscr R_+ $ and $ \mathscr T $ (under some assumptions guaranteeing existence and
uniqueness of solutions to the wave equation) is an elementary general fact -- see Proposition \ref{p:sim}. 
The point here is these operators are related to the stationary problem from Theorem \ref{t:1} and have correct mapping properties. At this stage, we provide the codes for solving the wave equation (see the Appendix and Figure~\ref{f:1}) but not the comparison with the scattering matrix. It is an interesting open question, relevant to physics problems considered in \cite{horse} and 
references given there, to analyse quantitative properties of $ R_+ $ and $ T $ in asymptotic regimes of the parameters.
The numerical experiments (which an interested reader is invited to perform using code in the Appendix) indicate interesting phenomena which should investigated.

\smallsection{Acknowledgements} 
JG acknowledges support from EPSRC grants EP/V001760/1 and EP/V051636/1, the Leverhulme Trust under Research Project Grant  RPG-2023-325, and the ERC under the Synergy grant PSINumScat 101167139. and 
MZ from
the Simons Foundation under a ``Moir\'e Materials Magic" grant. We are also grateful to S.A.R. Horsley for 
discussions about physical motivation and to Bryn Davies for useful references. The second author would also like to thank 
H. Ammari for his 
hospitality in Zurich where he first learned about time dependent materials with memory, and 
Univerisity College, London, for providing support during his visit there.

\section{General theory} 
\label{s:gent} 

We discuss a general definition of a scattering operator for spatially localised perturbations of the 1+1 wave equation.
We then define a general class of perturbations with time dependence and memory and prove existence and
uniqueness for the corresponding wave equation. These preliminary results are not surprising 
but we not seem to have a ready to use reference covering existence and uniqueness of the Cauchy problem for operators 
described in \S \ref{s:intr} and for the more general ones given in \eqref{eq:defP}. We present a detailed argument in our specific case noting that it generalises to higher dimensions.

\subsection{An abstract scattering operator}
\label{s:genr}

The following elementary result shows that we can define the scattering matrix for very general spatially localised perturbations (see \S \ref{s:revs} for a review 
of the standard theory in our context). 

Suppose that $ P :  \mathscr D' ( \mathbb R^2 ) \to \mathscr D' ( \mathbb R^2 ) $  has the property that
\begin{equation}
\label{eq:supp} 
 u \in \mathscr D' ( \mathbb R^2 ) , \ \supp u \cap  ( \mathbb R \times [ - R , R ] )  = 
 \emptyset \ \Longrightarrow \ P u = D_x^2 u . 
 \end{equation}
More informally we can state this as 
\[  \forall \, t \in \mathbb R  \   \supp u (t, \bullet )  \cap [ - R , R ] = \emptyset \ \Longrightarrow P u ( t, x ) = D_x^2 u ( t , x ) .\] 
 We do not assume that $ P $ is linear here. For a general class of linear operators $ P$ relevant to this note, see
 \S \ref{s:class}. 
 
Under this assumption on $ P $ we have the following simple fact:
\begin{prop}
\label{p:sim}
Suppose that $ u \in \mathscr D' ( \mathbb R^2 ) $ solves $ ( D_t^2 - P ) u = 0 $ and 
\[   \begin{gathered} u|_{ t < 0 }  = \kappa_+^* g |_{ t < 0 } , \ \ \  \kappa_\pm  : \mathbb R^2 \to \mathbb R , \ \ \kappa_\pm  ( t, x ) := t \mp x , \\
g \in \mathscr D ' ( \mathbb R ) , \ \ \supp g \subset ( - \infty, - R ) . \end{gathered}  \]
Then there exist 
\[  G , F \in \mathscr D' ( \mathbb R ) , \ \ 
\supp G , \supp F \subset ( -R , \infty), 
\]
 such that
\begin{equation} 
\label{eq:upm}  u|_{ x <  - R } =  \kappa_+^* g |_{ x < -R  } + \kappa_-^* G |_{ x < -R } , \ \ \
u|_{ x > R } = \kappa_+^* F |_{x >  R } .\end{equation} 
\end{prop} 
Less formally the hypothesis reads as 
\begin{equation}
\label{eq:u2g}   u ( t, x ) = g ( t - x ), \ \  t < 0  , \ \ \ \supp g \subset ( - \infty, - R ) , 
\end{equation}
and the conclusion as 
\begin{equation}
\label{eq:g2FG}  u ( t , x ) = \left\{ \begin{array}{ll}  g (t - x ) + G ( t + x  ) , \ \  x < - R , \ \ \supp G \subset ( - R , \infty ) , \\
F ( t - x ) , \ \ x > R,  \ \ \supp F \subset (- \infty, R  ). \end{array} \right. \end{equation}
This means that under the assumption \eqref{eq:supp} and provided we have existence and uniqueness to solutions 
of $ ( D_t ^2 - P ) u = 0 $, we have scattering maps:
\[    g \mapsto \mathscr T g := F , \ \ \  g \mapsto \mathscr R_+ g = G , \]
where $ \mathscr T $ and $ \mathscr R_+ $ are transmission and reflection maps. (We have similar definitions for waves $ g $ approaching from the right.)  As explained in \S \ref{s:intr}  this paper describes mapping properties and basic structure of $ \mathscr T $ and $ \mathscr R_+ $ for more specific perturbations and relates them to stationary scattering theory.

\begin{proof}[Proof of Proposition \ref{p:sim}]
We first proceed by pretending that $ u $ is a function and take $ r > R $. If 
we write $ z = t - x $ and $ w = x + t $ so that for $ v ( z, w ) = u ( x, t ) $ we have $ \partial_z \partial_w v = 0 $ for 
$  \pm (w - z) > 2 r $.  In particular, $ \partial_z \partial_v  v ( z, w ) = 0 $ for $  - z > 2 r - w $, that is 
$ \partial_w v ( z, w ) = f ( w ) $, $ w > 2 r +  z $ for some $ f$ defined for all values of $ w $.  But that means that
\[  v ( z , w ) = v ( z , 2r  + z ) + \int_{ 2 r +  z }^0 f ( y ) dy  + \int_0^w f( y ) dy, \   \ \text{ \ for $ w - z  >  2r $,} \]
 that is  $ v ( z , w ) = 
G_+ ( z ) + F_+ ( w ) $ for $ w - z  > 2r $. Similarly, $ v ( z , w ) = G_- ( z ) + F_- ( w ) $ for $ w- z  < - 2r $. 
This argument applies to distributions by using \cite[Theorem 3.1.4${}'$]{H1} and the fact that 
 the restriction of $ v $ to $ L := \{ ( z, 2r + z ) : z \in \mathbb R \}$ is well defined as for $  w -z > 2 R  $, 
$ \WF ( v ) \subset \{ ( z, w; \zeta , \omega ) : \zeta \omega = 0 \} $ which (away from the zero section) is disjoint from 
$ N^* L = \{ ( z, 2 r - z ,  \zeta ,  \zeta ) : z , \zeta \in \mathbb R \} $ -- see \cite[Corollary 8.2.7]{H1}. Since $ r > R $ 
is arbitrary we can replace $ r $ with $ R $ in our conclusion. 

We now need to prove that $ G_+ \equiv 0 $ and $ G_- = g $.  To see the first claim we note that \eqref{eq:u2g} gives 
$ F_+  ( z ) + G_+ ( w ) = 0 $ for $ w - z = 2 x > 2 R $ and $ w + z = 2t < 0 $. In particular, $ G_+ ' ( w ) = 0 $ for 
$ 2 R + z < w < - z $, and as $ z $ is arbitrary, this shows that $ G_+ $ is constant. We can absorb that constant into $ F_+$
and hence have $ G_+ \equiv 0 $. 

To see that $ G_- = g $, we note that \eqref{eq:u2g} gives  $ G_- ( z ) - g ( z )  + F_- ( w ) = 0 $ for $ w - z  = 2x < -2R $ and 
$ w + z  = 2t < 0 $. Hence $ G_- ' ( z ) - g' ( z ) =0 $ for $  w <  -z < -2R - w $. Since $ w $ is arbitrary, this means that
$ G_- ( z )-  g ( z ) $ is constant and we can make it $ 0 $ by changing $ F_- $. 
\end{proof}

\subsection{A class of permittivities}
\label{s:class}

The model we consider is the wave equation with permittivity which depends on time but also has memory in the sense of being an operator. In this section w do not assume spatial localisation and consider a generalisation of \eqref{eq:newave}:
\begin{equation}
\begin{gathered} 
\label{eq:defP}  
  \mathscr{P}: = \partial_t  \varepsilon  \partial_t - \partial_x^2,  \ \ \ x \in \mathbb R , \ \ \ \varepsilon v ( t, x ) := v ( t, x ) + B v ( t, x)  , \\
B v ( t, x ) := \int_{-\infty}^t B (x,  t , t - t' ) v (t' , x ) dt'. 
\end{gathered}
\end{equation}
We typically want to solve the problem 
\begin{equation}
\label{eq:inho}   \mathscr{P} v = F ( t , x ) , \ \ \ \supp F \subset ( - R , \infty ) \times \mathbb R , \ \ \  v ( t, x )|_{ t < - R  } = 0 . 
\end{equation}
We make the following assumptions on $ B $:and that there exists $ m $ and $ m_0 $ such that 
\begin{equation}
\label{eq:assB} 
\forall \, k , \ell \, \exists \, C_{ k\ell } \ \  |\partial^k_t \partial^\ell_s B ( x, t , s ) |\leq C_{k\ell} \langle t \rangle^{{m_0}} \langle s\rangle ^{m} . 
\end{equation}
Later, but not in this section, we assume that we have a localisation in space, that is,  for $ R > 0 $ independent of $ t $ and $ s $
\begin{equation}
\label{eq:suppB}
    \supp B ( \bullet, t , s ) \subset ( - R , R ) 
\end{equation}

As an example we can take 
\begin{equation}
\label{eq:Bex}   B ( x, t, s  ) = V ( x ) e^{-\alpha t^2 }  e^{- \gamma s } s^m , \ \ \  \alpha, \gamma  \geq 0, \ \  \ V \in L^\infty
( \mathbb R ) ,  
\end{equation} 
noting that the case of $ \alpha = \gamma = 0 $, $ m = 1 $,  gives $ \mathscr P = \partial_t^2 - \partial_x^2 + V ( x ) $. Another way of writing the operator
$ B $ in \eqref{eq:Bex} is as a pseudodifferential operator: 
\begin{equation}
\label{eq:FTo2t}  
 \begin{gathered}   B = b ( x, t, D_t ) , \ \ \    b ( x, t , \tau ) := V ( x ) e^{ - \alpha t^2 } i^m  ( \tau + i \gamma )^{-m} , \\
 a ( t, D_t ) h := \int_{\mathbb R } a ( t, \tau ) \widehat h ( \tau ) d \tau , \ \ \  
h ( t ) = \frac{1}{2\pi}\int_{\mathbb R } \widehat {h} ( \omega ) e^{ - i \omega t } d \omega . 
\end{gathered} \end{equation}
 We note here that our convention for Fourier transform in time is non-standard but leads to cleaner formulas in our setting. 
 
 The most relevant case for us is \eqref{eq:newave}. For that we take
 \begin{equation}
 \label{eq:ourB}
 \begin{gathered} 
 B ( x,t, s ) = \gamma^{-1} (  e^{-s \gamma } -1 ) a ( x, t ) , \ \ \gamma > 0 ,  \\ a \in L^\infty_x C^\infty_t, \ \ \supp a \subset [ - R, R ] \times [ - T , T ].  \end{gathered} 
 \end{equation}
 In that case we have 
 \[    m_0 = - \infty, \ \  m = 1 . \]
 (This means that in \eqref{eq:energy} below we can take $ \gamma ( T ) \equiv 1 $.)


Since we do not know a ready to use reference covering existence and uniqueness of the Cauchy problem for $ \mathscr P $ in \eqref{eq:defP} we present a detailed argument in our specific case noting that it generalises to higher dimensions.

\subsection{Existence and uniqueness}
\label{s:eu}

The result of \S \ref{s:genr} is applicable to operators of the form \eqref{eq:defP} thanks to the following
\begin{prop}
\label{p:exu}
Under the assumptions \eqref{eq:defP} and \eqref{eq:assB}, 
for $ F \in L^1 ( [0,T ]; L^2 ( \mathbb R ) ) $, there exists 
a unique
\begin{equation}
\label{eq:assu}  u \in C( {(-\infty,T])} ; H^1 ( \mathbb R ) ) \cap C^1 ( {(-\infty,T])} ; L^2 ( \mathbb R ) ) ,   \ \ \
\supp u\subset [0,T], 
\end{equation}
such that 
\[   \mathscr P u = F  \text{ on $  ( 0, T ) \times \mathbb R  $}  
\]
Moreover, there exists $ C_0, \lambda > 0  $ (independent of $ T $) such that for $ 0 \leq t \leq T $. 
\begin{equation}
\label{eq:energy} 
\begin{gathered} 
 \| u ( t, \bullet ) \|_{ H^1 ( \mathbb R ) } + \| \partial_t u ( t, \bullet ) \|_{L^2 ( \mathbb R) }  \leq
 C_0 
\big\|  e^{ \lambda \gamma ( T )  ( T - t' ) }  F ( t' , x ) \big\|_{ L^2 (0,T)_{t'}\times \mathbb{R}_x  ) } , \\
\gamma ( T  ) := 1+ \langle T\rangle^{{m_0}+1}+ \langle T \rangle^{{m_0}+m+1} . 
\end{gathered} 
\end{equation}
\end{prop}
To obtain uniqueness and \eqref{eq:energy} we will use
\begin{lemm}
\label{l:energy}
Suppose that \eqref{eq:assu} is satisfied and that in addition,  
\[  F := \mathscr P u \in L^1 ( [ 0 , T ] ; L^2 ( \mathbb R )  ). \]
Then \eqref{eq:energy} holds.
\end{lemm} 
\begin{proof}
This is done by an adaptation of the usual energy estimate based on the  energy identity (see for instance 
\cite[(2.4.2)]{H3} for a very general version):
\begin{equation*}
\begin{split} & \partial_t (e^{ - 2 \lambda t } \tfrac12 ( | u_t |^2 + |\nabla_x u |^2 + |u|^2 ) ) -  \nabla_x \cdot ( \Re \nabla_x u \bar u_t )
+ \lambda e^{ - 2 \lambda t }  ( | u_t |^2 + |\nabla_x u |^2 + |u|^2 ) 
 \\
\ \ \ & = - 2 \Im e^{ - 2 \lambda t } \mathscr{P} u \overline {D_t u } - e^{ - 2 \lambda t } D_t B D_t u \overline{ D_t u}- {e^{-2\lambda t}\mu^2 u \overline{D_tu}}  + e^{ - 2 \lambda t }
  \Re u \bar u_t .  
\end{split} 
\end{equation*}
Suppose $u\in H^1((-\infty, T)\times \mathbb{R})$ with $\supp u\subset [0,T]$ and $\mathscr{P}u\in L^2$. Then, integrating the energy identity on $(-\infty,T)\times \mathbb{R}$, using the divergence theorem
{(in higher dimensions)}, the fact that {the form of 
$B $ in ~\eqref{eq:defP} implies $\supp BD_t u\subset ( - \infty ,T]$,} and that for $\supp v\subset [0,T]$,
\begin{align*}
&\int _{-\infty}^T \Big\|e^{-\lambda t}D_t\int_0^{t}B(x,t,t-t')e^{\lambda t'}v(t',x)dt'\Big\|_{L^2_x}^2dt\\
&\leq C\int _{0}^T \langle t\rangle ^{2{m_0}}\|v(t,\bullet)\|_{L^2_x}^2dt+ C\int_0^T\langle t\rangle^{2{m_0}}(1+\langle t\rangle^{2m}) \|v\|_{L^1((0,t);L^2_x)}^2dt\\
&\leq C \gamma ( T )^2 \|v\|_{L^2}^2,
\end{align*}
where $ \gamma ( T ) $ is defined in \eqref{eq:energy}. 
From this we obtain
\begin{align*}
&e^{-2\lambda T}\tfrac{1}{2}(\|u_t(T)\|_{L^2(\mathbb{R})}^2 +\| u ( T) \|_{H^1(\mathbb{R})}^2)+\lambda ( \|e^{-\lambda t}u\|_{L^2((-\infty,T);H^1(\mathbb{R}))}^2+ \|e^{-\lambda t}u_t\|_{L^2((-\infty,T)\times \mathbb{R})}^2) \\
&=-2\Im \langle e^{-\lambda t}\mathscr{P}u,e^{-\lambda t}D_tu\rangle_{(-\infty,T)\times \mathbb{R}}-\langle e^{-\lambda t}D_tBe^{\lambda t}e^{-\lambda t}D_t u,e^{-\lambda t}D_tu\rangle_{L^2 ((-\infty,T)\times \mathbb{R}) }\\
&{\qquad-\langle e^{-\lambda t}\mu^2 e^{\lambda t}e^{-\lambda t} u,e^{-\lambda t}D_tu\rangle_{(-\infty,T)\times \mathbb{R}}} +\Re \langle e^{-\lambda t}u,e^{-\lambda t}\partial_t u\rangle\\
&\leq C\|e^{-\lambda t}f\|_{L^2((-\infty,T)\times \mathbb{R})}^2+(C_0 \gamma ( T ) ^2 +1)\|e^{-\lambda t}D_t u\|_{L^2((-\infty,T)\times \mathbb{R})}^2
\\ &\qquad
+\tfrac{1}{2}\|e^{-\lambda t}u\|_{L^2((-\infty, T)\times \mathbb{R}) }^2.
\end{align*}

Hence, taking $\lambda$ large enough, and moving the two right-most terms to the left-hand side, we obtain~\eqref{eq:energy}
(with $ \lambda \gamma ( T )  $ replacing $ \lambda $). 
\end{proof}

\begin{proof}[Proof of Proposition \ref{p:exu}]
Uniqueness is immediate from Lemma \ref{l:energy} and \eqref{eq:energy}. It remains to show 
existence. For that we will use the free wave group, $ U ( t ) : H^{s} ( \mathbb R ) \times H^{s-1} ( \mathbb R ) \to 
H^{s} ( \mathbb R ) \times H^{s-1} ( \mathbb R ) $:
\begin{equation}
\label{eq:freeU}
\begin{split} 
U(t) & = :\begin{pmatrix} \cos (D_x t)&  {\sin(D_xt)}/{D_x}\\ -D_x\sin (D_xt)&\cos(D_xt)\end{pmatrix} \\
& = e^{t L} , \ \ \ \ \  L:=\begin{pmatrix}\ \ 0&I\\ - D_x^2&0\end{pmatrix} , \ \ \ D_x := (1/i) \partial_x .  
\end{split}
\end{equation}
We fix $T_-\geq 0$ and for 
\begin{equation}
\label{eq:propfw}  \mathbf w \in C^0((-\infty,T_-);H^1\times L^2), \ \ \ 
\supp \mathbf w\subset [0,T] \times \mathbb R ,
\end{equation} 
define  a sequence 
\begin{equation}
\label{eq:propvn}   \mathbf v_n ( t ) \in  C^0((-\infty,T);H^1\times L^2) , \ \  n = -1, 0 , \cdots , \ \ \
\supp \mathbf v_n  \subset [0,T] \times \mathbb R  , \end{equation}
inductively as follows:
$$
{\bf v}_{-1}:=\begin{cases} \mathbf w(t)&t\leq T_-, \\ \mathbf w(T_-)&T_-\leq t\leq T . \end{cases}
$$ 
Then for $n\geq 0$, we again define $ \mathbf v_n $ differently in different ranges of $ t $. For 
$ t \leq T_- $ we put
$   \mathbf v_n ( t ) := \mathbf w ( t ) $, 
while for 
$ T_- \leq t \leq T$, 
\begin{equation}
\label{eq:Duha}
{\bf v}_n(t):= U(t-T_-)\mathbf w(T_-)+\int_{T_-}^tU(t-s)\Bigg(\begin{pmatrix} 0\\F(s)\end{pmatrix} -\begin{pmatrix}0&0\\\mu^2&\partial_tB\end{pmatrix}{\bf v}_{n-1}(s)\Bigg)ds . 
\end{equation}
Then \eqref{eq:propvn} holds and 
\begin{gather*}
(\partial_t-L){\bf v}_n=\begin{pmatrix} 0\\F\end{pmatrix} -\begin{pmatrix}0&0\\\mu^2&\partial_tB\end{pmatrix}{\bf v}_{n-1},\quad t\in (T_-,T)\qquad {\bf v}_{n}(T_-)=\mathbf w(T_-). 
\end{gather*}
In particular, 
\begin{equation}
\label{eq:dtL}
\begin{gathered}
(\partial_t-L)({\bf v}_n-{\bf v}_{n-1})= -\begin{pmatrix}0&0\\\mu^2&\partial_tB\end{pmatrix}\big({\bf v}_{n-1}-{\bf v}_{n-2}\big),\\ t\in (T_-,T), \ \ \ {\bf v}_{n}(T_-)-{\bf v}_{n-1}(T_-)=0.
\end{gathered}
\end{equation}
We now observe that \eqref{eq:defP} and \eqref{eq:assB} show that for $ v \in C^0 ( ( - \infty , T ) ; L^2 )$, 
$ \supp v \subset [0, T]\times \mathbb R $ 
\begin{equation}
\label{eq:DtBv} 
\begin{split}  \|  \partial_t B v ( s ) \|_{ L^2 ( \mathbb R ) }^2 
& \leq \int_{\mathbb R } \left| \int_{0}^s \langle s \rangle^m \langle s - t' \rangle^{{m_0}} v ( t' , x ) dt' \right|^2 dx \\
& \leq \langle s \rangle^{2m  }(1+\langle s\rangle^{2{m_0}}) \| v \|_{ L^1 ( ( 0, s ) ; L^2 ) )}^2 . \end{split} \end{equation}
Since $ \mathbf v_{n-2} (t) - \mathbf v_{n-1} ( t) = 0 $ for $ t < T_- $ it follows from \eqref{eq:Duha} that
\[ \begin{split} 
 \|{\bf v}_{n} (t)-{\bf v}_{n-1}(t) \|_{H^1\times L^2} \leq  &  C\int_{T_-}^t(1+|t-s|)(\|{\bf v}_{n-1}(s)-{\bf v}_{n-2}(s) \|_{H^1\times L^2}\\
 & \ \ \ \ \ \ \ \ +
\langle s \rangle^{m}(1+\langle s\rangle^{{m_0}}) \|{\bf v}_{n-1}-{\bf v}_{n-2}\|_{L^1(T_-,s);H^1\times L^2})ds,
\end{split} \]
for some constant $ C $.

Using $T_-+\delta\leq T+1$, we see that for $ \delta < 1 $,    there is a new constant $ C_1 $ depending on {$ T $ but not on $T_-$}, 
\[ \begin{split}
& \|{\bf v}_{n}-{\bf v}_{n-1}\|_{L^\infty((T_-,T_-+\delta);H^1\times L^2)} \leq \\
& \ \ \ \ {C(1+\langle T_-+\delta\rangle^{m}(1+\langle T_-+\delta\rangle^{{m_0}})\delta)}  \|{\bf v}_{n-1}-{\bf v}_{n-2}\|_{L^1((T_-,T_-+\delta);H^1\times L^2)}
\leq \\
& \ \ \ \ C_1 \delta \|{\bf v}_{n-1}-{\bf v}_{n-2}\|_{L^\infty((T_-,T_-+\delta);H^1\times L^2)}\
\end{split} \]
Taking $\delta < 1/C_1 $, we concluded that ${\bf v}_n$ is a Cauchy sequence in
 $L^\infty((T_-,T_-+\delta);H^1\times L^2)$.  Since 
\[ {\bf v}_{n} ( t )= \mathbf w ( t) \ \text{ for } \  t \leq T_-,  \ \ \ {\bf v}_n\in C^0([T_- ,T_-+\delta] ;H^1\times L^2),
\ \ \mathbf v_n ( T_- ) = \mathbf w ( T_- ) , 
\]
${\bf v}_n$ converges to ${\bf v}\in C^0((-\infty ,T_-+\delta];H^1\times L^2)$ and 
 ${\bf v}|_{(-\infty,T_-)}= \mathbf w|_{(-\infty,T_-)}$.  
 
 In particular we have convergence $ \mathbf v_n (t ) \to \mathbf v ( t ) $ in the sense of distributions on 
 $ ( T_- , T_- + \delta ) $, and hence, from \eqref{eq:Duha}, in the sense of distributions, 
 \[    \partial_t \mathbf v  = L \mathbf v  -\begin{pmatrix}0&0\\\mu^2&\partial_tB\end{pmatrix} {\bf v} + 
 \begin{pmatrix} 0 \\ F \end{pmatrix}, \ \ \text{ on $ ( T_- , T_- + \delta ) \times \mathbb R $.} \]
Putting $ \mathbf v = [ v^1 , v^2 ]^t $,  this gives  $ \partial_t v^1 = v^2 \in C^0 ( [ T_- , T_- + \delta ] ; L^2 ) 
 $. Hence $ u := v^1 $, satisfies
 \begin{equation}
 \label{eq:propu} 
 \begin{gathered}  u \in C^0 ( ( -\infty  , T_+ + \delta ]; H^1 ) \cap C^1 ( -\infty , T_- + \delta ] ) ,  \\
 u( T_- ) = w^1 ( T_- ) ,  \ \ \ \partial_t u ( T_- ) = w^2 ( T_- ) ,   \ \ \ \mathbf w =:[ w^1  , w^2 ]^t , \\
 \mathscr P u = F  \  \text{ in the sense of distributions on $ ( T_-, T_- + \delta ) $. } 
 \end{gathered}
 \end{equation}
The only conditions here are \eqref{eq:propfw} on $ w $ and and $\delta>0$. 
 
We can now pass from this small step procedure to finding 
\[ u\in C^0((-\infty,T);H^1(\mathbb{R})) \cap  C^1((-\infty, T);L^2(\mathbb{R} )) ,\ \ \ \supp u\subset [0,T],  \]
satisfying 
$
\mathscr{P}u=F $ on  $(0,T)
$.
To do this, we set $u_0=0$ and for $j\geq 1$ and use \eqref{eq:propu} to inductively obtain 
\begin{equation}
\label{eq:uj01}  u_j\in C^0((-\infty,j\delta] ;H^1(\mathbb{R})) \cap C^1 ((-\infty, j\delta] ;L^2(\mathbb{R})) , 
\end{equation}
satisfying 
\begin{equation}
\label{eq:propuj}
\begin{gathered}
\mathscr{P}u_j= F, \ \ \text{ distributionally on  }((j-1)\delta,j\delta)\times \mathbb{R}, \\
u_j ( ( j-1 ) \delta ) = u_{j-1} ( (j-1) \delta ) , \ \ \partial_t u_j ( ( j-1 ) \delta ) = \partial_t u_{j-1} ( ( j-1 ) \delta ) . 
\end{gathered}
\end{equation}
We claim that 
\begin{equation}
\mathscr{P}u_j=F \ \ \text{ distributionally on  }(0,j\delta).
\end{equation}
Indeed, suppose by induction that this holds for $ j $ replaced by $ j-1 $. Then, \eqref{eq:uj01} holds and
$$
\mathscr{P}u_{j}=F \ \ \text{ distributionally on  }((0,(j-1)\delta)\cup((j-1)\delta, j \delta).
$$
We we write $ \mathscr P = \partial_t^2 - \mathscr L u $, $ 
\mathscr L u := \partial_x^2 u - \partial_t B \partial_t u - \mu^2 u $, then
\[ G :=  \mathscr L u \in C^0 ( ( - \infty , j \delta ) ; H^{-1} ) ,  \]
and hence, with $ H := F + G $, 
\[ \partial_t^2 u = H \in  C^0 ( ( - \infty , j \delta ) ; H^{-1} ) 
\text{ distributionally on  }((0 ,(j-1)\delta)\cup((j-1)\delta, j \delta). \]
But the continuity of $ u $ and $ \partial_t u $ across $ (j-1) \delta $ shows that the equation holds distributionally
on $ ( 0, j \delta ) $. (The argument reduces to showing that if $ g , G \in C^0 $ and  
$ \partial_t g = G $, on $ \mathbb R \setminus \{ 0 \} $ then $ \partial_t g = G $ on $ \mathbb R$.
But that follows from testing against $ \varphi ( t ) ( 1- \chi ( t/\varepsilon )) $ where $
\varphi, \chi \in C_{\rm{c}}^\infty ( \mathbb R ) $, $ \chi \equiv 1 $ near $ 0 $, and using continuity
of $ g $ and $ G $ while letting $ \varepsilon \to 0 $.)
Taking $J\geq \delta^{-1}T$ and setting $u=u_J$ completes the construction of $u$. \end{proof}

\section{An operator valued scattering matrix}
\label{s:opers}

In this section we prove Theorem \ref{t:1}. To motivate the construction and to link it to standard theory 
we first review the case compactly supported time independent potentials (the case $ \alpha = \gamma = 0 $ 
and $ m = 1 $ in \S \ref{s:class}) . The wave equation setting (unlike the more common Schr\"odinger equation
setting) makes the the scattering matrix appear in a very clean way. 
 
 \begin{figure}
\includegraphics[width=15cm]{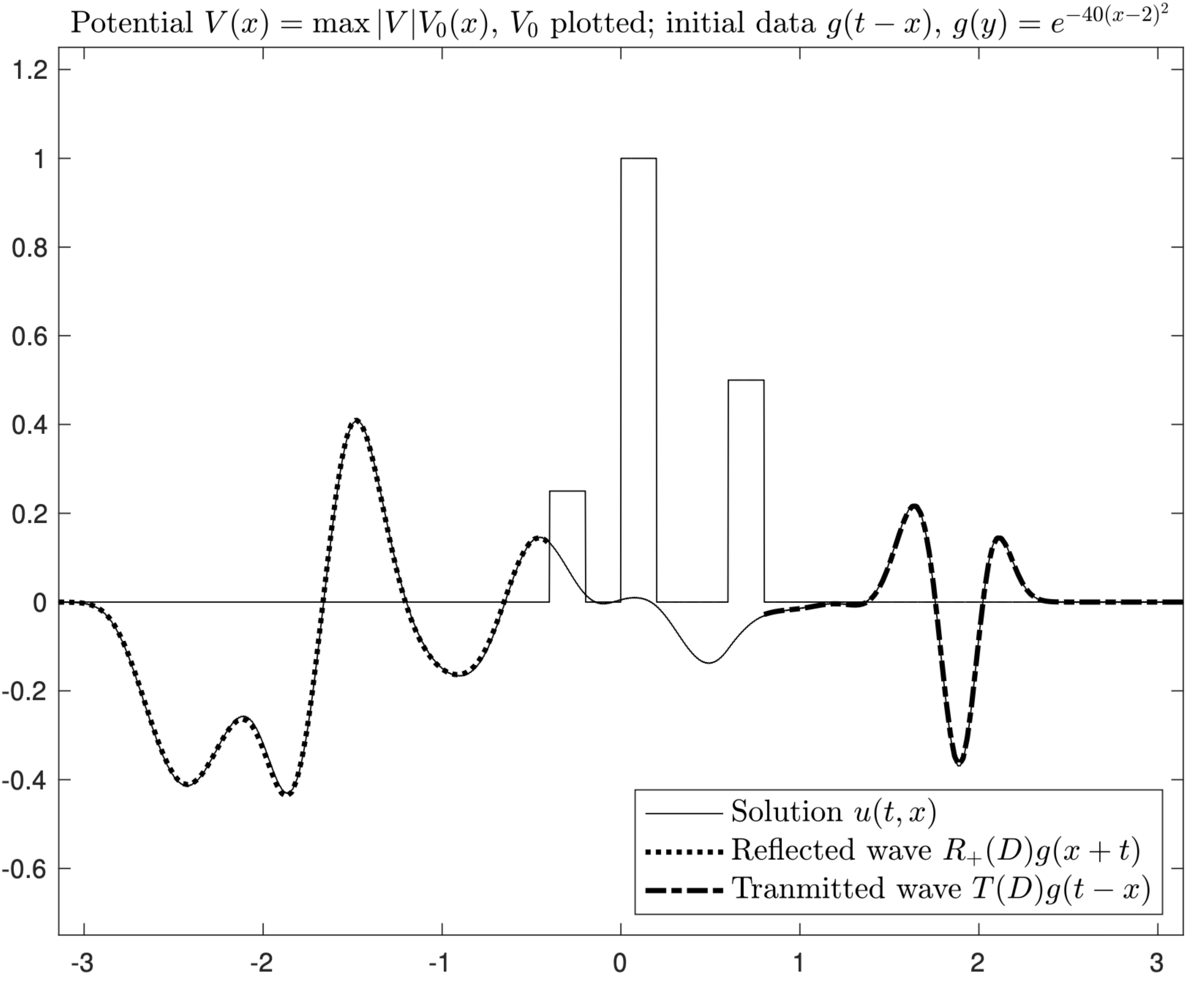}
\caption{\label{f:2} An illustration of Proposition \ref{p:1Dwave}: the solution to $ ( \partial_t^2 - \partial_x^2 + V ( x ) ) u = 0$,
$ \supp V \subset [ - R , R ] $, 
with $ u ( t, x ) = g ( t - x ) $ for $ t \ll -1 $ is given by $ g ( t - x ) + R_+ ( D) g ( t + x )$ for $ x \leq - R $ and
$ T ( D ) g ( t - x) $ for $ x \geq R $ (for all times; $ t = 4 $ shown). 
An animated version is available at \url{https://math.berkeley.edu/~zworski/wave_pot.mp4}.}
\end{figure}
 
\subsection{Review of standard scattering matrix}
\label{s:revs}
Suppose that $ V \in L^\infty ( \mathbb R; \mathbb R ) $ satisfies $ \supp V \subset [ - R , R ] $. Any solution to 
\begin{equation}
\label{eq:Schr}  ( D_x^2  +V ( x ) -\lambda^2 ) u = 0 , 
\end{equation}
 satisfies
\[  u ( x ) = \left\{ \begin{array}{ll} A_+ e^{ i \lambda x } + B_- e^{-i \lambda x } , & x > R; \\
A_- e^{ i \lambda x } + B_+ e^{ - i \lambda x }, & x < - R. \end{array} \right. \]
Taking the Wronskian of $ u $ and $ \bar u $ for $ \lambda \in \mathbb R \setminus  0 $ shows that $ |A_-|^2 + |B_-|^2 = 
| A_+|^2 + |B_+ |^2 $, and hence 
\begin{equation}
\label{eq:defS}  S ( \lambda ) : \begin{pmatrix} A_- \\ B_- \end{pmatrix} \mapsto  \begin{pmatrix} A_+\\ B_+ \end{pmatrix} , \ \ \
S ( \lambda ) = \begin{pmatrix} T ( \lambda ) & R_- ( \lambda ) \\ 
R_+ ( \lambda ) & T ( \lambda ) \end{pmatrix} , \end{equation}
is well defined. It is called the {\em scattering matrix}, $ T ( \lambda ) $ is the transmission coefficient and 
$ R_\pm  ( \lambda ) $ are the reflection coefficients. They extend to {\em meromorphic} functions in $ \mathbb C $ 
satisfying 
\begin{equation}
\label{eq:unitTR}  
\begin{gathered} T ( \lambda ) \overline {T ( \bar \lambda )} + R_\pm ( \lambda )\overline{ R_\pm ( \bar \lambda ) } = 1, \ \ \ 
T ( \lambda) \overline{ R_\pm ( \lambda) } + R_\pm ( \lambda ) \overline{ T ( \bar \lambda ) } = 0 , \\
T ( - \lambda ) = \overline{ T ( \bar \lambda ) } , \ \ \  R_\pm ( - \lambda ) = \overline{ R_\pm ( \bar \lambda ) } , 
\end{gathered} 
\end{equation}
and
\begin{equation}
\label{eq:expbound}  \prod_{j=1}^N \frac{ |\lambda - i \mu_j | } { | \lambda + i \mu_j | }  | T ( \lambda ) | \leq e^{ 2 R \Im \lambda } , \ \ \   \prod_{j=1}^N \frac{ |\lambda - i \mu_j | } { | \lambda + i \mu_j | } | R_\pm ( \lambda ) | \leq e^{ 2 R \Im \lambda } , \ \ \ \Im \lambda \geq 0 ,
\end{equation}
where $  \mu_j > 0 $ and $ - \mu_j^2 $, $ j = 1, \cdots, N $, are the eigenvalues of $ D_x^2 + V ( x ) $ -- 
see \cite[\S 2.4, Exercise 3.14.10]{res}.  The normalised eigenfunctions 
$ w_j  \in L^2 ( \mathbb R ; \mathbb R ) $ satisfy
\begin{equation}
\label{eq:nev}   ( D_x^2 + V ( x) ) w_j = - \mu_j^2 w_j , \ \  \| w_j \|_{L^2 } = 1 , \ \  w_j ( x )|_{ \pm x > R }  = a_j^\pm e^{ \mp \mu_j x }  , \ \ 
\mu_j > 0 .    
\end{equation}

When $ V \geq 0 $ there are no eigenvalues and  in that case we can consider multiplication by  $ T $ and $ R_\pm $ on 
Hardy spaces:
\begin{equation}
\label{eq:mapping1}
\begin{gathered}  H_{\alpha} := \{ f \in \mathscr O ( \{ \Im \lambda > 0 \} ) : \sup_{\sigma > 0 } 
e^{ - 2 a \sigma } \int_{\mathbb R } | f ( \lambda + i \sigma ) |^2 d \lambda < + \infty  \}, \\
H_{\alpha} \ni f (\lambda ) \mapsto T ( \lambda ) f ( \lambda ) \in H_{ a + 2 R } , \ \ \
H_{\alpha} \ni f (\lambda ) \mapsto R_\pm ( \lambda ) f ( \lambda ) \in H_{ a + 2 R } .
\end{gathered} 
\end{equation}
We remark that the spaces are isometric as $ f ( \omega ) \mapsto e^{ i a \omega } f ( \omega ) $ provides
a unitary map between $ H_{\alpha} $ and $ H $.

The relation to the wave equation is given in the following
\begin{prop}
\label{p:1Dwave}
Suppose that $ g \in \mathscr S' ( \mathbb R ) $ and that $ \supp g \subset ( R , \infty ) $. If 
\begin{equation}
\label{eq:wave1} 
( D_t^2 - D_x^2 - V ( x ) ) u ( t, x ) = 0 , \ \ t \in \mathbb R,  \ \ \  u ( t , x ) = g ( t - x  ) , \ \ t \leq 0 , \end{equation}
then, in the notation of \eqref{eq:defS} and \eqref{eq:nev},
\begin{equation}
\label{eq:scat2wave}
u ( t , x ) = \left\{ \begin{array}{ll}  g ( t -x  ) + M_- g (x+t ) , & x < - R , \\
M_+  g ( t- x ) , & x > R ,  \end{array} \right. \end{equation}
where
\[ \begin{gathered}   M_- g ( x ) = R_+ ( D ) g ( x ) + \sum_{ j=1}^N \langle g , w_j \rangle a_j^{-} e^{ \mu_j x } , \\
M_+ g ( x ) = T ( D) g ( x ) + \sum_{ j=1}^N \langle g , w_j \rangle a_j^{+} e^{ \mu_j x } , \end{gathered}  \]
and where $ T $ and $ R_+ $ are defined in \eqref{eq:defS}.
\end{prop}

\noindent
{\bf Remark.} At first it is not clear that the expression in \eqref{eq:scat2wave} has the desired support properties. 
To see this when there are no eigenvalues, we first
note that by the Paley--Wiener Theorem $ \widehat g ( \lambda) $ (with the convention of \eqref{eq:FTo2t}) 
is holomorphic for $ \Im \lambda > 0 $ and
satisfies the bound $ \langle \lambda\rangle^N e^{  -(R + \delta  ) \Im \lambda } $, $ \delta > 0 $, there. In view of \eqref{eq:expbound} we then have 
\[     T ( \lambda ) \widehat g ( \lambda ) ,  \  R_+ ( \lambda ) \widehat g ( \lambda ) = \mathcal O  ( \langle \lambda \rangle^N )
e^{ (R - \delta ) \Im \lambda },  \ \ \ \Im \lambda > 0 . \]
The Paley--Wiener Theorem then shows that $ \supp T ( D ) g , \supp R_+ ( D ) g \subset ( - R ,  \infty ) $ and
\begin{gather*}   \supp \indic_{ - \infty , - R )} ( \bullet )R_+ (D ) g ( \bullet +t  ) \cap ( - \infty , R ) \subset \{ x: - t - R < x < - R \} , \\
\supp \indic_{ ( R, \infty ) } T( D ) g ( t - \bullet  ) \cap ( R , \infty ) \subset \{ R < x < t + R \}. \end{gather*} 
In particular for $ t \leq 0 $ the scattering terms both vanish. 
When there are eigenvalues, $ - \mu_j^2 $, then the entries of $ S ( \lambda ) $ have poles $ i \mu_j $, $ \mu_j > 0 $.  A contour deformation then provides the needed cancellation. 

\begin{proof}
We use the spectral representation of the wave propagator constructed using distorted plane waves, that
is solutions to $ ( D_x^2 + V - \lambda^2 ) e_\pm = 0 $ satisfying
\begin{equation}\label{eq:epm}
e_{\pm}(x,\lambda) =\begin{cases}
T (\lambda)e^{\pm i \lambda x} &  {\text{for }} \ \ \pm x >  R , \\
e^{\pm i \lambda x} +R_\pm (\lambda )e^{\mp i\lambda x}&  {\text{for }} \ \ \pm x < -R.
\end{cases}
\end{equation}
Then, for $ g \in C_{\rm{c}}^\infty (\mathbb R ) $, $ \supp g \subset ( R, \infty ) $, 
\begin{equation*}
\begin{split} 
u ( t, x ) & = \sum_{ j = 1}^N ( \langle g , w_j \rangle \cosh  \mu_j t - \langle g' , w_j \rangle \mu_j^{-1} \sinh \mu_j t )  w_j ( x) 
\\  & + \frac{1}{ 2 \pi } \int_0^\infty \int_{\mathbb R } ( e_+ ( x, \lambda ) \overline{ e_+ ( y, \lambda ) } + 
 e_- ( x, \lambda ) \overline{ e_- ( y, \lambda ) } ) ( \cos t \lambda - \lambda^{-1} \sin t \lambda \partial_y ) g ( - y ) dy d \lambda . 
\end{split}  \end{equation*}
In view of the support properties of $ g $ and the form of $ w_j $ for $ x < - R $ we have
\[  \langle g' , w_j \rangle = \int g' ( x ) a_j^- e^{ x \mu_j } dx = -  \mu_j \int g ( x ) a_j^- e^{ x \mu_j } dx = 
- \mu_j \langle g , w_j \rangle . \]
Hence the contribution of $ w_j $'s to $ u ( t , x ) $ is given by 
\[  \sum_{ j = 1}^N e^{ t \mu_j } \langle g, w_j \rangle w_j ( x ) =  \sum_{ j = 1}^N a_j^\pm  \langle g, w_j \rangle 
e^{  \mu_j  ( t \mp x ) } , \ \  \pm x > R . \] 
To see the contribution of the continuous spectrum, 
we now use the expressions for $  \overline{ e_\pm ( y, \lambda ) } $ 
 for $ y < - R $:
\begin{equation}
\label{eq:y2la} 
 \begin{split}
& E_+ g ( \lambda , t ) := \int_{\mathbb R } \overline{ e_+ ( y, \lambda ) } ( \cos t \lambda - \lambda^{-1} \sin t \lambda \partial_y ) g (- y ) dy =
 \widehat g ( \lambda ) e^{ - i \lambda t } + \overline{ R_+ ( \lambda ) } \widehat g ( - \lambda ) e^{ i \lambda t } ,\\
& E_- g ( \lambda , t ) := \int_{\mathbb R } \overline{ e_- ( y, \lambda ) } ( \cos t \lambda - \lambda^{-1} \sin t \lambda \partial_y ) g ( -y ) dy =
  \overline{ T ( \lambda ) }  \widehat g ( - \lambda ) e^{  i \lambda t }  .
\end{split}
\end{equation} 
Then for $ x < -  R $, 
\[  \begin{split} e_+ ( x, \lambda ) E_+ g ( \lambda , t ) & =  ( e^{i \lambda x} +R_+  (\lambda )e^{- i\lambda x} ) (  \widehat g ( \lambda ) e^{ - i \lambda t } + \overline{ R_+ ( \lambda ) } \widehat g ( - \lambda ) e^{ i \lambda t } ) \\
& = \widehat g ( \lambda ) e^{ i \lambda ( x - t ) } +  |R_+ ( \lambda ) |^2 \widehat g ( -\lambda ) e^{
i \lambda ( t - x ) }  \\ & \ \ \ + R_+ ( \lambda ) e^{ - i \lambda ( x + t) } \widehat g ( \lambda )  +
R_+ ( - \lambda ) e^{ i \lambda ( x + t ) } \widehat g ( - \lambda )  e^{ i \lambda ( t - x ) } , \\
 e_- ( x, \lambda ) E_- g ( \lambda , t ) & = | T ( \lambda ) |^2 \widehat g ( - \lambda ) e^{ - i \lambda ( x - t ) } .
\end{split} \]
Using \eqref{eq:unitTR} in the expression for $ u ( t, x ) $ gives \eqref{eq:scat2wave} for $ x < - R $. 
Similar arguments give the expression for $ x > R $. 
\end{proof}

\subsection{Construction of $ u( x, \omega ) $ in Theorem \ref{t:1}}
\label{s:constr}

To motivate the construction of $ u ( x, \omega )$ we recall one way to do it in the case of \eqref{eq:Schr}.
For $ V \in L^\infty_{\rm{comp}} ( \mathbb R ; [0, \infty ) ) $, $ \supp V \subset ( - R, R ) $, 
 we consider  
the resolvent $ R_V ( \omega ) := ( D_x^2 +V - \omega^2 )^{-1} : L^2 (\mathbb R )  \to L^2 ( \mathbb R ) $ which is holomorphic for $ \Im \omega > 0 $. 
When $ V \not \equiv 0 $ (note we assumed that $ V \geq 0 $),  it is defined $ L^2_{\rm{comp}} ( \mathbb R )  \to L^2_{\rm{loc}} ( \mathbb R ) $ for $ \Im \omega = 0 $. (See \cite[Theorem 2.7]{res}.) 

For $ F $ with $ \supp F \subset ( - R_1, R_1 ) $, $ R < R_1 $, 
\begin{equation}
\label{eq:RV}  R_V ( \omega ) F ( x ) = 
\left\{ \begin{array}{ll}  A_- (\omega )   e^{ - i \omega x }, 
& x < -R_1  , \\ 
A_+ ( \omega )  e^{ i \omega x }  , & x > R_1 , \end{array} \right. \end{equation}
which is the meaning of being outgoing for 1D for compactly supported perturbations. (Here $ A_\pm $ depend on $ F $ and $ \omega $.) To construct a solution of $ ( D_x^2 + V - \omega^2 ) u = 0  $ satisfying \eqref{eq:scat2} we 
choose $ \rho \in C_{\rm{c}}^\infty ( \mathbb R ; [ 0 , 1 ] ) $ supported in $ ( - R_1 , R_1 ) $, $ R_1 > R $, and 
equal to $ 1 $ near the support of $ V $. We then put 
\[  u ( x, \omega ) = ( 1 - \rho ( x ) ) e^{ i \omega x } f ( \omega ) + R_V ( \omega ) [ D_x^2 , \rho ] ( e^{ i \bullet \omega }  f ( \omega ) ) ( x ) ,  \]
Then \eqref{eq:RV} shows that \eqref{eq:scat2} holds. We will mimic this strategy and construct $ R_V ( \omega ) $ now acting on spaces of functions of both $ x $ and $ \omega $. The key is the existence of $ R_V  $
in $ \mathscr H_{\alpha} $ which will be established in \S \ref{s:none}.

We start by recording some simple facts: 

\begin{lemm}
\label{l:1} 
For  $ R_0 ( \omega ) = ( D_x^2 - \omega^2 )^{-1} $, $ \Im \omega > 0 $, the outgoing resolvent, and $0 \leq s\leq 2 $, define 
\begin{equation}
\label{eq:defsR}   \mathscr R_{s} : f ( x, \omega ) \mapsto { \frac{\omega}{(\omega+i)^s}}  R_0 ( \omega )  f ( \bullet, \omega ) ( x ) .
\end{equation}
Then, in the notation of \eqref{eq:defHca},  for $ \rho \in C_{\rm{c}}^\infty ( \RR  ) $, and any $ \alpha \in \mathbb R $, 
\begin{equation}
\label{eq:l1}  
\begin{gathered} 
\| \rho \mathscr R_s \rho \|_{ \mathscr H_{\alpha}  \to \mathscr H^{\blue{s}}_{\alpha}} \leq C , \\
\rho  \mathscr R_s \rho f ( x, \omega )  := \rho ( x ) \omega ( \omega + i )^{-s}  R_0 (\omega ) (  \rho ( \bullet ) f ( \bullet ,\omega ) ) ( x ) .
\end{gathered}
 \end{equation}
\end{lemm}
\begin{proof} 
This follows immediately from he definition \eqref{eq:defHca} and the explicit formula for $ R_0 ( \omega  )$ as an operator  $ L^2_{\rm{comp}} 
( \mathbb R) \to H^2_{\rm{loc}} ( \mathbb R ) $:
\[ {\frac{\omega}{(\omega+i)^s}}R_0 ( \omega ) g ( x ) =\frac {i}{ 2{(\omega+i)^s}}  \int_{\mathbb R } e^{ i \omega |x-y| } f ( y ) dy .\]
(See~\cite[Theorem 2.1]{res} for estimates on $\|\rho R_0\rho\|_{L^2\to H^s}$.)
\end{proof}

The next lemma gives a compactness result:
\begin{lemm}
\label{l:2} 
Suppose that $ A_0 :=  a ( x, D_\omega ) ( \omega + i \gamma)^{-1}$, $ \gamma > 0 $,  and the assumptions in \eqref{eq:newave} hold. Then, for any $\alpha \in \mathbb R $, $ A_0 : L^2 ( \mathbb R^2 ) \to L^2 ( \mathbb R^2 ) $ defines a bounded operator
\begin{equation}
\label{eq:A0} 
A_0 : \mathscr H_{\alpha} \to \mathscr H_{\alpha} ,
\end{equation}
and, in the notation of Lemma \ref{l:1}, 
\begin{equation}
\label{eq:l2}
A_0 \mathscr R_{0} \rho : \mathscr H_{\alpha} \to \mathscr H_{\alpha}  \ \text{ is a compact operator.}
\end{equation}
\end{lemm}
\begin{proof}
To establish \eqref{eq:A0} we need to show that $ a(x, D_\omega) $ extends to an operator on 
$ \mathscr H_{\alpha} $. The holomorphy can be seen by applying the Cauchy Riemann operator and integration by parts 
justified by the rapid decay of $ \tau \mapsto \widehat{ a ( x, \bullet ) } ( \tau ) $ as $ \Re \tau \to \infty $. 

Since $ x \mapsto a ( x, t ) $ is compactly supported and bounded with values in $ C_{\rm{c}}^\infty ( \mathbb R ) $ (and the support in $ t $ is uniformly bounded) it is enough to show that 
for $ b \in C_{\rm{c}}^\infty ( \mathbb R ) $, $ b ( D_\omega) $ extends to an operator on $ \mathscr H_{\alpha} $, defined in 
\eqref{eq:mapping1}. 

We will need slightly more for \eqref{eq:l2}  and with this in mind we show that for any $s\in \mathbb{R}$,  $ (\omega+i)^s b ( D_\omega)(\omega+i)^{-s} $ extends to a bounded operator on $\mathscr{H}_{\alpha}$. 

In fact, recalling \eqref{e:aAction} we have, 
\begin{align*} &[ (\omega+i)^sb ( D_\omega )(\omega+i)^{-s} f ] ( \lambda + i \sigma  ) \\
&\ \ \ \ = \frac{1}{2\pi}\int (\lambda+i\sigma+i)^s\hat b ( \lambda -   \omega' ) (\omega'+i\sigma+i)^{-s}f ( \omega' + i \sigma ) d \omega' , 
\end{align*}
and, since $\hat{b}$ is rapidly decaying,
\begin{align*}
\sup_{\lambda} \int |\lambda +i\sigma +i|^s|\hat{b}(\lambda-\omega)||\omega +i\sigma+i|^{-s}d\omega &\leq C\\
\sup_{\omega} \int |\lambda +i\sigma +i|^s|\hat{b}(\lambda-\omega)||\omega +i\sigma+i|^{-s}d\lambda& \leq C.
\end{align*}
Thus, by the Schur test for boundedness,
\begin{align*} & \sup_{\sigma > 0 } e^{ - 2 \sigma \alpha } \int \left| \frac{1}{2\pi}\int (\lambda+i\sigma+i)^s\hat b ( \lambda - \omega ' ) (\omega'+i\sigma+i)^{-s}f ( \omega' + i \sigma  ) d \omega' \right|^2 d \lambda\\
& \ \ \ \ \ \leq C\sup_{\sigma > 0 } e^{ - 2 \sigma \alpha } \int
| f ( \omega' + i \sigma  ) |^2 d \omega' ,
\end{align*}
that is,  $ (\omega+i)^s b ( D_\omega)(\omega+i)^{-s}  $ is bounded on  $ \mathscr{H}_{\alpha}$. In particular, \eqref{eq:A0} follows.

For compactness, suppose that $\{f_n\}_{n=1}^\infty$ is bounded in $\mathscr{H}_{\alpha}$. Then, using Lemma \ref{l:1}, $\{\rho\mathscr{R}_{s}\rho f_n\}_{n=1}^\infty$ is bounded in $\mathscr{H}_{\alpha}^{s}$ for for $0\leq s\leq 2$. Moreover,  using the equality
$$
(\omega+i)^{\frac{1}{2}}A_0\mathscr{R}_0\rho f_n=(\omega+i)^{\frac{1}{2}}a_0(x,D_\omega)(\omega+i)^{-\frac{1}{2}}(\omega+i)^{}(\omega+i\gamma)^{-1}\rho\mathscr{R}_{\frac{1}{2}}\rho f_n,
$$
together with the facts that
$(\omega+i)^{\frac{1}{2}}a_0(x,D_\omega)(\omega+i)^{-\frac{1}{2}}:\mathscr{H}_{\alpha}\to \mathscr{H}_{\alpha}$ (proved above) and 
$ (\omega+i+i\sigma)/(\omega+i\gamma+i\sigma) | \leq 1+ 1/\gamma $,  $ \omega \in \mathbb R $, $\sigma > 0 $,  we have 
\begin{equation}
\label{e:omegaDecay}\sup_n\|(\omega+i)^{\frac{1}{2}}A_0\mathscr{R}_0\rho f_n\|_{\mathscr{H}_{\alpha}}<\infty
\end{equation}
Fixing $ \sigma $ we want to extract a convergent subsequence of
$ A_0\mathscr{R}_{0}\rho f_{n_{\sigma,k}}(\bullet,\bullet+i\sigma) $ in 
$ L^2_xL^2_\omega$. To apply Rellich's theorem we need improved regularity in $ x $ and $ \omega $ and decay in both. In $ x $ the decay comes from compact support property of the cut-off function $ \rho $, in $ \omega $ from the factor $ (\omega + i)^{\frac{1}{2}} $ in \eqref{e:omegaDecay}. The regularity improvement in $ x $ is a consequence of the boundedness of $\{
\rho\mathscr{R}_{{1}/{2}}\rho f_n,  n \in \mathbb N \} $  in $\mathscr{H}_{\alpha}^{{1}/{2}}$ 
 given in Lemma  \ref{l:1}. Compact support in $ D_\omega $ provides smoothness in $ \omega $

We conclude that for all $ \sigma \geq 0 $ (recall that the singularity
at $ \omega = 0 $ is removed in \eqref{eq:defsR}), there exists a subsequence $ n_{\sigma, k } $ such that
$$
e^{-\sigma \alpha}A_0\mathscr{R}_{0}\rho f_{n_{\sigma,k}}(\bullet,\bullet+i\sigma)\overset{L^2_xL^2_\omega}{\to} e^{-\sigma \alpha}g_\sigma.$$
Now, by a diagonal argument, we may find $n_k$ such that
$$
e^{-\ell \alpha}A_0\mathscr{R}_{0}\rho f_{n_{k}}(\bullet,\bullet+i\ell)\overset{L^2_xL^2_\omega}{\to}  e^{-\ell \alpha} g_\ell
$$
for all $\ell=0,1,\dots$. Since $ \zeta \mapsto A_0 \mathscr R_0 \rho f_{n_k} ( \bullet, \zeta) $ are holomorphic in $\Im \zeta \geq 0 $,
we have by the Phragm\'en--Lindel\"{o}f principle, 
\[ \begin{split}
& \sup_{0<\sigma <j}\|e^{i \alpha (\omega+i \sigma)}[A_0\mathscr{R}_{0}\rho (f_{n_{k}}-f_{n_{k'}})](x,\omega+i\sigma)\|_{L^2_\omega L^2_x }\\
& \ \ \ \ \ \ \ \ \leq \max_{\sigma\in\{0,j\}}\|e^{i \alpha (\omega+i\sigma)}[A_0\mathscr{R}_{0}\rho (f_{n_{k}}-f_{n_{k'}})](x,\omega+i\sigma)\|_{L^2_x L^2_x}.
\end{split}
\]
We now recall \eqref{e:omegaDecay} to see that 
$ \omega \mapsto (\omega+i\sigma +i)^{\frac{1}{2}}e^{i\alpha(\omega+i\sigma)}A_0\mathscr{R}_0\rho f_n ( \omega + i \sigma ) $ is uniformly bounded in $ L^2_\omega L^2_x $ 
and hence 
$$
\sup_{j<\sigma <\infty}\|e^{i(\alpha(\omega+i\sigma)}[A_0\mathscr{R}_{0}\rho (f_{n_{k}}-f_{n_j})](x,\omega+i\sigma)\|_{L^2_\omega L^2_x}\leq C j^{-\frac{1}{2}}
$$
Hence, we obtain a Cauchy sequence in $ \mathscr H_{\alpha} $, 
$
A_0\mathscr{R}_0\rho f_{n_k}\overset{\mathscr{H}_\alpha}{\to} g,
$
which concludes the proof.
\end{proof}

We now solve  $ P v = f $ where $ P $ is given in \eqref{eq:statwave}:
\begin{prop}
\label{p:resolve}
For  $ \alpha \in\mathbb{R}$ and $R_1>R$, let  $ f \in \mathscr H_{\alpha} $ and $ \supp f ( \bullet, \omega ) \subset ( - R_1 , R_1 ) $ for all $\omega\in\mathbb{C}_+$. Then there is a unique $v\in\omega^{-1}\mathscr{H}_{\alpha}$ such that 
$$
(D_x^2-\omega^2+A(x))v(x,\omega)=f(x,\omega).
$$
Moreover, for any 
$ \rho, \rho_1 \in C_{\rm{c}}^\infty ( \mathbb R ; [ 0, 1 ] )  $ such that $ \rho \equiv 1 $ near
 the support of $ \rho_1 $ and $ \rho_1 \equiv  1 $ on a neighbourhood of $ ( - R_1, R_1 ) $,
 $$
v= R_0(\omega)\rho_1(I+A(x)R_0(\omega)\rho)^{-1}\rho_1f.
$$
\end{prop}
\begin{proof}
 We will solve solve 
\[  (  D_x^2 - \omega^2 +   A( x ) ) v ( x, \omega ) = f ( x, \omega ),  \]
so that \eqref{eq:RV} holds for $ v ( x, \omega )$ and
\begin{equation}
\label{eq:uHardy}  \rho \in C^\infty_{\rm{c}} ( ( - R , R ); [ 0,1 ] ) \  \Longrightarrow \ 
 \rho v \in \omega^{-1}\mathscr H_{ a } . \end{equation}
 To start
we note that   
\[ ( 1 - \rho ) A = 0 \ \text{ and } \ ( I + A ( x ) R_0 ( \omega ) ( 1 - \rho ) )^{-1} = ( I - A ( x ) R_0 ( \omega ) ( 1 - \rho ) ) . \] 
Hence, 
\[  \begin{split}  P & = D_x^2 - \omega^2 + A ( x ) = ( I + A ( x ) R_0 ( \omega ) )( D_x^2 - \omega ^2) \\
& = 
( I + A ( x ) R_0 ( \omega ) ( 1 - \rho ) )  ( I + A ( x ) R_0 ( \omega ) \rho ) ( D_x^2 - \omega^2 ) ,
\end{split}  \]
and solving $ P u = f = \rho_1 f $ is equivalent to solving
\begin{equation}
\label{eq:u2f1}   ( I + A ( x ) R_0 ( \omega ) \rho ) ( D_x^2 - \omega^2 ) v = ( I + A( x ) R_0 ( \omega ) ( 1 - \rho ) ) \rho_1 f =
\rho_1 f. \end{equation}

To continue, we use the following crucial proposition which will be proved 
in Section \ref{s:none}.
\begin{prop}
\label{p:noker}
For $ A$ given in \eqref{eq:statwave}, 
\begin{equation}
\label{eq:noker}
\ker_{\mathscr H_{\alpha} }  ( I + A ( x ) R_0 ( \omega ) \rho ) = \{ 0 \} . 
\end{equation}
\end{prop}

Lemma \ref{l:2} shows that at $ A ( x ) R_0 ( \omega ) \rho : \mathscr H_{\alpha} \to \mathscr H_{\alpha} $ is compact and hence $ ( I + A ( x ) R_0 ( \omega ) \rho ) :  \mathscr H_{\alpha} \to \mathscr H_{\alpha}  $ is a Fredholm operator of index 0. Proposition~\ref{p:noker} then shows that we have an inverse
\begin{equation}
\label{e:bounded} ( I + A ( x ) R_0 ( \omega ) \rho ) ^{-1} : \mathscr H_{\alpha} \to \mathscr H_{\alpha} .\end{equation}
Going back to \eqref{eq:u2f1} and recalling that $\rho_1 f\in \mathscr{H}_a$ we define
\begin{equation}
\label{eq:f2u2}  v = R_0 ( \omega ) ( I + A ( x ) R_0 ( \omega ) \rho)^{-1} \rho_1 f, 
\end{equation}
and notice that
$$
( I + A ( x ) R_0 ( \omega ) \rho)^{-1} \rho_1 f=-A(x)R_0(\omega)\rho ( I + A ( x ) R_0 ( \omega ) \rho)^{-1} \rho_1 f+\rho_1 f.
$$
Hence, the support properties of $f$ and $A$ imply that 
$$
( I + A ( x ) R_0 ( \omega ) \rho)^{-1} \rho_1 f=\rho_1( I + A ( x ) R_0 ( \omega ) \rho)^{-1} \rho_1 f,
$$
which together with~\eqref{eq:f2u2} completes the proof of the proposition.
\end{proof}

\begin{proof}[Proof of Theorem~\ref{t:1}]
{\bf Existence:} Let $f\in  ( \omega + i )^{-\blue{1}}H_{\alpha}$, $R_1>R$, $\rho,\rho_1\in C_c^\infty(\mathbb{R})$ with $\rho\equiv 1$ near $[-R,R]$, $\supp \rho_1\subset (-R_1,R_1)$ and $\rho_1\equiv 1$ near $\supp \rho$.  

Since $f\in ( \omega + i )^{-\blue{1}} H_{\alpha}$ and $\rho\in C_c^\infty(-R_1,R_1)$, 
$$
F:=f(\omega)[D_x^2,\rho](e^{i\omega x})\in \mathscr{H}_{a+R_1}.
$$
 By Proposition~\ref{p:resolve} there is a unique $v\in \omega^{-1}\mathscr{H}_{a+R_1}$ such that 
$$
(D_x^2-\omega^2+A(x))v=F,
$$
and 
\begin{equation*}
v=R_0(\omega)\rho_1(I+A(x)R_0(\omega)\rho)^{-1}\rho_1 F .
\end{equation*}
Setting $u:= (1-\rho(x))e^{i\omega x}f(\omega)+v$, we have
$$
(D_x^2-\omega^2+A(x))u=0.
$$
Define 
\begin{align}
Tf (\omega)&:= f(\omega)+ \frac{i}{2\omega} \int e^{-iy\omega}\left[(I+AR_0\rho)^{-1}\left( \rho_1f(\omega)[D_{\bullet}^2,\rho](e^{i\omega \bullet})\right) \right](y,\omega)dy,\label{e:T}\\
R_+f(\omega)&:= \frac{i}{2\omega} \int e^{iy\omega} \left[(I+AR_0\rho)^{-1} \left( \rho_1f(\omega)[D_{\bullet}^2,\rho](e^{i\omega \bullet})\right) \right](y,\omega)dy\label{e:R+}.
\end{align}
Then Theorem~\ref{t:1} follows from the form of $R_0(\omega)$, the estimates~\eqref{e:bounded} and~\eqref{eq:l1}, the fact that 
$$\supp \{x\,:\exists \, \omega\in\mathbb{C}_+, (x,\omega)\in \supp (I+AR_0\rho)^{-1}\rho_1 F \}\subset (-R_1,R_1),$$
and that $R_1>R$ is arbitrary.

\noindent {\bf Uniqueness:} It is enough to show that for any $R>0$, $\alpha>0$, any solution, $v\in \omega^{-1}\mathscr{H}_{\alpha}$ to $Pv=0$ with 
$$
v(x,\omega)=\begin{cases} g_-(\omega)e^{-i\omega x},&x<-R\\
g_+(\omega)e^{i\omega x},&x>R
\end{cases}
$$
and $g_\pm \in \omega^{-1}\mathscr{H}_{\alpha}$,  satisfies $v\equiv 0$. 

To see this we recall that $ P v = 0 $ means that
\begin{equation}
\label{eq:A2A}
(D_x^2-\omega^2)v= -A(x)v\in (\omega+i)^{-1}\mathscr{H}_{\alpha} \subset \mathscr H_\alpha
\end{equation}
Therefore, since $v(\omega,\bullet)$ is outgoing and $\supp A(\bullet)\subset (-R,R)$,  
\begin{equation}
\label{e:vForm}
v=-R_0(\omega)A(x)v
\end{equation}
Hence
$$
A(x)v=-A(x)R_0(\omega)A(x)v,
$$
so that for $\rho\in C_c^\infty$ with $\supp (1-\rho)\cap [-R,R]=\emptyset$, 
$$
0=A( x) v + A(x)R_0(\omega)A(x)v=(I+A(x)R_0(\omega)\rho)A(x)v.
$$
In view of the inclusion in \eqref{eq:A2A}, we can now use Proposition~\ref{p:noker} sto see that $A(x)v=0$.
Together with~\eqref{e:vForm}, this implies $v=0$, completing the proof of uniqueness.
\end{proof}

\subsection{Non-existence of purely outgoing solutions}
\label{s:none}
We will now prove Proposition \ref{p:noker}. Suppose that, in the notation of Proposition 
\ref{p:noker}, 
\begin{equation}
\label{eq:wker}   ( I + A ( x ) R_0 ( \omega ) \rho ) w = 0 , \ \ \  w \in \mathscr H_{\alpha} .  
\end{equation}
Since $ \rho w = w $, 
\begin{equation}
\label{eq:uout}
u := R_0 ( \omega ) w = -  R_0 ( \omega )  A ( x ) R_0 ( \omega ) \rho w , 
\ \ \  ( D_x^2 - \omega^2 + A( x ) ) u  = 0 . \end{equation}
Moreover, since 
\begin{equation}
\label{e:whereIsU}R_0(\omega):\mathscr{H}_{\alpha}\to \omega^{-1}(\mathscr{H}_{\alpha}\cap \langle \omega\rangle\mathscr{H}_{\alpha}^1\cap\langle \omega\rangle^2\mathscr{H}_{\alpha}^2)\end{equation}
by Lemma~\ref{l:1}, $ u ( x, \bullet ) \in  \mathscr{O} ( \mathbb C_+ ) $, 
and for $ \sigma > 0 $,
\[ \int_{\mathbb R} | u ( x, \omega + i \sigma ) |^2 d \omega < \infty. \]
We want to show that $ u ( x, \omega + i \sigma ) \equiv 0 $, $ \sigma > 0 $ and for that we use a variant of 
a positive commutator argument. Define
\begin{gather}  \langle u, v \rangle_\sigma := e^{ - 2\sigma \alpha }  \int u ( x, \omega + i \sigma ) \overline{ v ( x, \omega + i \sigma ) } d x d \omega , \nonumber
\\
 P_\sigma := D_x^2 - ( \omega + i \sigma )^2 + a( x, D_\omega ) \frac{ \omega + i \sigma }{
\omega + i \sigma + i \gamma } .\label{e:Psigma} 
\end{gather}
Notice that for $u\in\mathscr{H}^2_\alpha$,  $P_\sigma (u|_{\Im \omega =\sigma})=(Pu)|_{\Im \omega =\sigma}$ where the action of $P$ on $u$ is explained in~\eqref{e:aAction}. 
For $u\in \omega^{-3/2} \mathscr{H}_{\alpha}^2$, we compute  
\begin{equation}
\label{eq:Puu} 
 \begin{split} 
\langle P_\sigma u , ( \omega + i \sigma ) u \rangle_\sigma & = 
\langle ( \omega - i \sigma ) D_x u , D_x u \rangle_\sigma - \langle  ( \omega^2 + \sigma ^2 ) 
( \omega + i \sigma ) u , u \rangle_\sigma \\
& \ \ \ \ \ \ \ \ \ + 
\langle a ( x, D_\omega ) \frac{ \omega + i \sigma}{ \omega + i \sigma + i \gamma } u, 
( \omega + i \sigma ) u \rangle_\sigma .  \end{split} \end{equation}
We now take the imaginary part to obtain 
\[ \begin{split} \Im \langle P_\sigma u, ( \omega + i \sigma ) u \rangle_\sigma & = \, 
- \sigma \big( \| D_x u \|^2_\sigma + \| | \omega + i \sigma | u \|^2_\sigma \big) 
\\
& \ \ \ \ + 
\Im \langle a ( x, D_\omega ) \frac{ \omega + i \sigma}{ \omega + i \sigma + i \gamma } u, 
( \omega + i \sigma ) u \rangle_\sigma  \\
& \leq - \sigma \big( \| D_x u \|^2_\sigma + \| | \omega + i \sigma | u \|^2_\sigma \big) \\
& \ \ \ \  + 
 \| a ( x, D_\omega ) \|_{ \mathscr H_{\alpha} \to \mathscr H_{\alpha} } 
\Big\|  \frac{ \omega + i \sigma}{ \omega + i \sigma + i \gamma } u\Big\|_\sigma \| |\omega + i \sigma |
 u \|_\sigma \\ & \leq 
  - \sigma ( \| D_x u \|^2_\sigma + \| | \omega + i \sigma | u \|^2_\sigma ) + 
  C \| u \|_{\sigma } \| | \omega + i \sigma | u \|_{\sigma } \\
  & \leq - \tfrac12 \sigma^3 \| u \|_{\sigma }^2 , 
  . \end{split}  \]
  if $ \sigma $ is large enough. 

Since $u=R_0(\omega)w$ with $w\in\mathscr{H}_{\alpha}$, we use~\eqref{e:whereIsU} to obtain $u\in \omega^{-1}(\mathscr{H}_\alpha\cap \langle \omega\rangle\mathscr{H}_\alpha^1\cap \langle \omega\rangle^2\mathscr{H}_{\alpha}^2)$. Therefore, to justify multiplication by $ \omega $ and integration as above, we 
consider a modified version of \eqref{eq:Puu}: 
\[ \begin{split}  & \langle P_\sigma u , ( \omega + i \sigma ) \indic_{ |\omega | \leq T }  u \rangle_\sigma = 
\langle  ( \omega - i \sigma ) \indic_{ |\omega | \leq T }  D_x u , D_x 
u \rangle_\sigma  \\
& \ \ \ - \langle \indic_{ |\omega | \leq T }  ( ( \omega^2 + \sigma ^2 ) 
( \omega + i \sigma ) u , u \rangle_\sigma 
+ 
\langle a ( x, D_\omega ) \frac{ \omega + i \sigma}{ \omega + i \sigma + i \gamma } u, 
( \omega + i \sigma ) \indic_{ |\omega | \leq T }  u \rangle_\sigma ,
\end{split} \]
estimate the imaginary part as above to obtain (recall that $P_\sigma u=0$ from~\eqref{eq:uout})
$$
0= \langle P_\sigma u , ( \omega + i \sigma ) \indic_{ |\omega | \leq T }  u \rangle_\sigma \leq -\tfrac{1}{2}\sigma^3\|1_{|\omega|\leq T}u\|_{\sigma}^2.
$$
For each $\sigma$ large enough, sending $T\to \infty$ then implies
$u|_{\Im \omega=\sigma}\equiv 0$. Since $u$ is holomorphic in $\Im \omega>0$, this implies that $u\equiv 0$. 

\section{Proof of Theorem 2}

Let $R_1>R$ and $\rho \in C_c^\infty((-R_1,R_1))$ with $\rho\equiv 1$ near $[-R,R]$, $\rho_1\in C_c^\infty((-R_1,R_1))$ with 
\[ \supp \rho\cap \supp (1-\rho_1)=\emptyset.
 \] 
 We then define \[ u_1(t,x):=(1-\rho(x))g(t-x), \ \   u_2:=u-u_1 . \] 
 It follows that  $ u_2\equiv 0$  for $ t\ll -1$ and that
\[
\begin{aligned} 
(D_t^2-D_x^2-A)u_2 & =-(D_t^2-D_x^2-A)(1-\rho)g(t-x)\\
& =-[D_x^2,\rho]g(t-x)=:f(t,x), 
\end{aligned} 
\]
Proposition~\ref{p:exu} and \eqref{eq:ourB} imply that there is $C>0$ such that $\|u_2\|_{L^2(\mathbb{R}_x)}\leq Ce^{Ct}$ and hence for $\sigma>0$ large enough,
$$
\|\widehat{u}_2(\omega+i\sigma,x)\|_{L^2_{\omega,x}}<\infty,
$$
and there is $\sigma_0>0$ such that $\widehat{u}_2$ is holomorphic in $\{\Im \omega>\sigma_0\}$.
Moreover, (recall the convention~\eqref{eq:FTo2t}) since $f$ is smooth and compactly supported in time, there is $a>0$ such that $\widehat{f}\in \mathscr{H}_{\alpha}$. Thus, for $\sigma>0$ large enough,
$$
P_\sigma \widehat{u}_2(\omega+i\sigma,x)=-\widehat{f}(\omega+i\sigma,x),
$$
where $P_{\sigma}$ is defined in~\eqref{e:Psigma}.
Hence, for $\sigma>0$ large enough,
\begin{equation}
\label{e:ftU2}
\widehat{u}_2(\omega+i\sigma,x)=-\Big(R_0(\omega+i\sigma)\Big[(I+AR_0\rho)^{-1}\rho_1\widehat{f}\,\Big](\omega+i\sigma,\bullet)\Big)(x).
\end{equation}
(See~\eqref{eq:statwave} to~\eqref{e:aAction} for a description of the action of $A$.)
By Lemma~\ref{l:1}, this implies that for $0\leq s\leq 2$,
\begin{equation}
\label{e:uIsHereAgain}
\frac{\omega}{(\omega+i)^s} \widehat{u}_2(\omega,x)\in \mathscr{H}_{\alpha}^s.
\end{equation}
Now,  taking the inverse Fourier transform and using~\eqref{e:ftU2} 
$$
u_2(t,x)=-\int_{\Im z =\sigma}e^{iz t}\Big(R_0(I+AR_0 \rho)^{-1} \rho_1 \widehat{f} \,\Big)(z,x)dz. 
$$
Define 
$$
\Gamma_{r,\pm,\varepsilon}:=\pm r+ i[\varepsilon, \sigma],\qquad \Gamma_{\sigma,r}:=[-r,r]+i\sigma. 
$$
First, using~\eqref{e:uIsHereAgain} to see that $(\omega+i\sigma)\widehat{u}_2(\omega+i\sigma,x)\in L^2_{x,\omega}$, we obtain
\[ \begin{split} 
&  \Big\|\int_{|\omega|>R}e^{-i\omega t+\sigma t}\widehat{u}_2(\omega+i\sigma,x)d\omega\Big\|_{L^2_x}
\\ & \ \ \ \  \ \leq e^{\sigma t}\Big(\int_{|\omega|>r}|\omega+i\sigma|^{-2}d\omega\Big)^{1/2}\|(\omega+i\sigma)\widehat{u}_2(\omega+i\sigma,x)\|_{L^2_{x,\omega}}  \leq  Cr^{-1/2}.
\end{split} \] 

Next, since $(I+AR_0\rho)^{-1}:\mathscr{H}_{\alpha}\to \mathscr{H}_{\alpha}$ is bounded and 
$$
\|R_0(z)\|_{L^2\to L^2}\leq \frac{1}{|z|\Im z},\qquad \Im z>0,
$$
we have
$$ 
\sup_{\Im z>0}\|e^{-a\Im z}|z||\Im z|\big(R_0(I+AR_0\rho)^{-1}\rho_1\widehat{f}\,\big)(z,\bullet)\|_{L^2(\mathbb{R}_x)}<\infty,
$$
and hence
$$
\Big\|\int_{\Gamma_{R,\pm,\varepsilon}}\widehat{u}_2(z, \bullet )dz\Big\|_{L^2_x}\leq Cr^{-1}\log \varepsilon^{-1}.
$$
Deforming the contour and letting $ r \to \infty $, we obtain that for $\varepsilon>0$
$$
u_2(t,x)=-\frac{1}{2\pi}\int_{-\infty}^\infty e^{-i(\omega+i\varepsilon) t}R_0(\omega+i\varepsilon)[(I+AR_0\rho)^{-1}\rho_1\widehat{f}\, ](\omega+i\varepsilon)d\omega.
$$
Now, sending $\varepsilon\to 0^+$, we obtain
\begin{align*}
u_2(t,x)&=-\frac{1}{2\pi}\int_{-\infty}^\infty e^{-i\omega t}\frac{1}{\omega+i0}\omega R_0(\omega)[(I+AR_0\rho)^{-1}\rho_1\widehat{f}\, ](\omega, x)d\omega.
\end{align*}
The integral makes sense as a distributional pairing 
with $ ( \omega + i 0)^{-1} $ since, by Lemma~\ref{l:1}, $\omega R_0(\omega):\mathscr{H}_{a,\comp}\to \mathscr{H}_{a,\loc}$. 

Recalling the definition of $f$, we have
\begin{align*}
\widehat{f}(\omega,x)&=\int_{-\infty}^\infty e^{i\omega s}[D^2,\rho]g(s-x)ds
=[D^2,\rho]\int_{-\infty}^\infty e^{-i\omega (x-s)+i\omega x}g(s-x)ds\\
& =\widehat{g}(\omega)([D^2,\rho]e^{i\omega \bullet})(x).
\end{align*}
Thus,
\begin{align*}
u_2(t,x)
&=\frac{1}{2\pi}\int_{-\infty}^\infty e^{-i\omega t}\frac{1}{\omega+i0}\omega R_0(\omega)(I+AR_0\rho)^{-1}\rho_1\widehat{g}(\omega)([D^2,\rho]e^{i\omega \bullet})d\omega.
\end{align*}
Therefore, using the definition~\eqref{e:R+} for $x<-R_1$, 
\begin{align*}
u_2(t,x)&=\frac{1}{2\pi} \int_{-\infty}^\infty e^{-i\omega(t+x)} [R_+\widehat{g}(\bullet)](\omega)d\omega=[\mathscr{R}_+g](t+x)
\end{align*}
and, using the definition~\eqref{e:T} for $x>R_1$, 
$$
u_2(t,x)=\frac{1}{2\pi} \int_{-\infty}^\infty e^{-i\omega(t-x)} [T\widehat{g}(\bullet)-\widehat{g}(\bullet)](\omega)d\omega\\
=[\mathscr{T}g](t-x)-g(t-x)
$$
Hence,  $u=u_1+u_2$ satisfies
$$
u(t,x)=\begin{cases} [\mathscr{T}g](t-x)&x>R_1,\\ g(t-x)+[\mathscr{R}_+g](t+x)&x<-R_1.\end{cases}
$$
Since $R_1>R$ is arbitrary the Theorem follows.
\qed

\renewcommand{\theequation}{A.\arabic{equation}}
\refstepcounter{section}
\renewcommand{\thesection}{A}
\setcounter{equation}{0}

\section*{Appendix by Zhen Huang and Maciej Zworski}




\subsection{Mathematical setup} 

We consider the following special case of operators defined in \S \ref{s:class}:
\begin{equation}
\label{eq:newave1}  
 D_t^2 u  - a ( x, t ) \mathcal B D_t u - D_x^2 u  = 0 , \ \ \ D_t := \tfrac 1 i \partial_t, \ \ D_x =  \tfrac 1 i \partial_x . 
\end{equation}
where
\begin{equation}
\label{eq:defa}    a ( x, t ) \in L^\infty ( \mathbb R_x , \mathscr S ( \mathbb R_t ) ), \ \ \ \supp a ( \bullet, t ) \subset [ -R , R ] \end{equation} 
and $ \mathcal B $ is a memory term:
\begin{equation}
\label{eq:mem}
\mathcal B v ( t, x ) := \int_{-\infty}^t e^{ - \gamma ( t - t' ) } v ( t', x ) dt' , \ \ \ \gamma > 0 .
\end{equation}
An example in \eqref{eq:defa} in the spirit of \cite{horse} would be 
\[  a ( x, t ) = V ( x ) \exp ( - \alpha  ( t - t_0 ) ^2), \ \ \  \alpha > 0 ,  \ \ \ V ( x ) = A   \indic_{ [ -1, 1 ] } . \]
We also allow any step potential:
\[ V ( x ) = \sum_{ j=1}^J V_j   \indic_{ [ x_j, x_{j+1} ] }, \ \ \  x_1 < x_2 <  \cdots < x_{J+1} . \]
As the initial condition we take
\begin{equation}
\label{eq:defg}
u ( t , x)|_{ t \leq 0 }  = g ( x - t ) , \ \ \  g \in C^\infty_{\rm{c}} ( ( - \infty, - R )  , 
\end{equation}
and in practice we could take
\[ g ( x ) = \exp( - ( x - x_0 )^2 / \sigma ) \exp ( i \lambda ( x - x_0 ) ), \ \ \ 0 < \sigma \ll 1, \ \ \ x_0 \ll - R . \]

We can rewrite \eqref{eq:newave1} as a system
\begin{equation}
\label{eq:sys}
D_t \mathbf u = \begin{pmatrix} 0 & 1 \\
D_x^2 & a ( x, t ) \mathcal B \end{pmatrix} \mathbf u, \ \ \ \mathbf u ( 0 , x ) = 
\begin{pmatrix} g ( x ) \\ i g ' ( x ) \end{pmatrix} , \ \ \ \mathbf u ( t, x ) = \begin{pmatrix} u( t, x ) \\
D_t u ( t, x ) \end{pmatrix} .
\end{equation}

The code can be tested in the special case of no memory effect by setting $ \gamma  = 0 $, $ \alpha = 0 $, and $ A =  -  i V_0 $. Then
\[ \mathcal  B D_t u ( t, x ) = \frac1 i \int_{-\infty}^t \partial_t u (t', x ) dt' = - i u (t , x ) . \]
Hence, for our choice of parameters, 
\[  a ( x, t ) \mathcal B D_t u = - i a ( x, t ) u ( t ,x ) = V_0 \indic_{[-1,1]} ( x ) u ( t, x ) , \]
that is, our equation becomes
\[ D_t^2 u - ( D_x^2 + V ( x ) ) u = 0 , \ \ \ V (x) = V_0 \indic_{[-1,1]} . \]
For this there is a standard code for comparing solutions. 
\begin{algorithm*}
\begin{algorithmic}[1]
\State \textbf{Input:} Initial condition $\mathbf{u}_0 = (u_0, v_0)$, time step $\Delta t$, number of time steps $N$, uniform spatial grid points $x$.
\State \textbf{Output:} Solution $\mathbf{u}_n$ at times $t_n = n\Delta t$ for $n = 1, \ldots, N$

\State Initialize memory term $\boldsymbol{\mathcal{M}}_0$ using trapezoidal rule from $-\infty$ to $0$.
\State \textbf{Startup:} Compute $\mathbf{u}_2$ from $\mathbf{u}_1$ using forward Euler (one step).
\For{$n = 2$ to $N-1$}
    \State Update memory term:
    $$
    \boldsymbol{\mathcal{M}}_n = e^{-\gamma \Delta t} \boldsymbol{\mathcal{M}}_{n-1} + \Delta t v_n.
    $$
    \State Leapfrog update:
    $$
    u_{n+1} = u_{n-1} + 2\mathrm{i} \Delta t \, v_n, \quad v_{n+1} = v_{n-1} + 2\mathrm{i} \Delta t \, (D_x^2 u_n + a(x, t_n) \boldsymbol{\mathcal{M}}_n).
    $$
\EndFor
\end{algorithmic}
\caption{Leapfrog scheme for \eqref{eq:sys} with memory term \eqref{eq:mem}}\label{alg:memory}
\end{algorithm*}

\subsection{Numerical schemes}
Now we aim to write down a numerical scheme for  \eqref{eq:sys}, which we rewrite as 
$$
D_t\mathbf u =L\mathbf u + a(x,t)\int_{-\infty}^t C(t-s)\mathbf u(s)ds,
$$
where $$
L = \begin{pmatrix} 0 & 1 \\ D_x^2 & 0 \end{pmatrix}, 
\quad C(t-s) = \begin{pmatrix}
    0 & 0 \\ 0 & e^{-\gamma(t-s)} \end{pmatrix} .
$$
Let $\mathbf u_n(x) = \mathbf u ( t_n, x )$, where $t_n = n\Delta t$ and $\Delta t$ is the time step. Due to the simple exponential form of $C(t)$, we use recurrence to update the memory term:
$$
\boldsymbol{\mathcal M}_n: = \int_{-\infty}^{t_n} C(t_n-s)\mathbf u(s)ds = e^{-\gamma \Delta t}\int_{-\infty}^{t_{n-1}} C(t_{n-1}-s)\mathbf u(s)ds + \int_{t_{n-1}}^{t_n} C(t_n-s)\mathbf u(s)ds.
$$
For $n=0$, we evaluate $\boldsymbol{\mathcal M}_0  = \int_{-\infty}^{0} C(t_n-s)\mathbf u(s)ds$ using the trapezoidal rule. For $n\geq 1$, we have the approximate recurrence relation:
$$
\boldsymbol{\mathcal M}_{n} \approx e^{-\gamma \Delta t} \boldsymbol{\mathcal M}_{n-1} + \Delta t C(0) \, \mathbf u_n,
$$
where $\mathbf u_n$ is updated using the following leapfrog scheme: for $\mathbf u_n = (u_n, v_n)^T$,
\begin{align}
    u_{n+1} &= u_{n-1} + 2\mathrm i \Delta t \, v_n, \label{eq:lf_u}\\
    v_{n+1} &= v_{n-1} + 2\mathrm i \Delta t \, (D_x^2 u_n + a(x, t_n) \boldsymbol{\mathcal M}_n), \label{eq:lf_v}
\end{align}
where the memory term $\boldsymbol{\mathcal M}_n$ is updated via the recurrence relation at each step. A one-step startup procedure using the forward Euler scheme is required to compute $\mathbf u_1$ from $\mathbf u_0$:
\begin{align}
    u_1 &= u_0 + \mathrm i \Delta t \, v_0, \\
    v_1 &= v_0 + \mathrm i \Delta t \, (D_x^2 u_0 + a(x, t_0) \boldsymbol{\mathcal M}_0).
\end{align}
The leapfrog scheme is explicit and has second-order accuracy in time. The overall algorithm is in Algorithm~\ref{alg:memory}.

\subsection{Matlab codes}

This code produces the movie and the graphs shown in Figure~\ref{f:1}. 

\begin{Verbatim}[numbers=left]
function wave_memory(gamma,alpha,V0,x0,lambda)
% Computes a toy wave-like PDE with memory
% Usage: wave_memory(gamma,alpha,V0,x0,lambda)
% If arguments are omitted, reasonable defaults are used.
%
% The code evolves three parameter cases in parallel (three gammas/alphas).
% Subfunction Vnew(x,V0,x0) returns piecewise-constant potential.
% --- Defaults and input sanitizing ---
if (nargin < 1 ) gamma = [0,3,4]; end
if (nargin < 2 ) alpha = [0,2,10]; end
if (nargin < 3 ) V0 = 10*[40,-5,30]; end
if (nargin < 4 ) x0 = [-0.5,-0.25,0.25,0.5]; end
if (nargin < 5 ) lambda = 10; end
% Ensure alpha and gamma have length at least 3 (pad or truncate)
if numel(alpha) < 3, alpha = [alpha(:).' zeros(1,3-numel(alpha))]; end
if numel(gamma) < 3, gamma = [gamma(:).' zeros(1,3-numel(gamma))]; end
% --- Spatial / temporal grid ---
L = pi;         % domain [-L, L]
T = 3.65;        % final time
dx = 0.01;
dt = 0.005;     % dt should be smaller than dx for stability here
x = -L:dx:L;
t = 0:dt:T;
Nx = length(x);
Nt = length(t);
% Memory parameter and potential
t0 = 2;
V = Vnew(x,V0,x0);
% --- Initial condition (rename centre to avoid colliding with x0 input) ---
y0 = -2.0;
sigma = 0.05;
g_func = @(x1) exp(-((x1 - y0).^2) / sigma) .* exp(1i * lambda * (x1 - y0));
% derivative: ( -2*(x-y0)/sigma + i*lambda ) * g
g_prime = @(x1) ( -2*(x1 - y0)/sigma + 1i*lambda ) .* g_func(x1);
% Preallocate arrays for three cases
u1 = zeros(Nt, Nx); v1 = zeros(Nt, Nx);
u2 = zeros(Nt, Nx); v2 = zeros(Nt, Nx);
u3 = zeros(Nt, Nx); v3 = zeros(Nt, Nx);
u1(1,:) = g_func(x);   v1(1,:) = 1i*g_prime(x);
u2(1,:) = g_func(x);   v2(1,:) = 1i*g_prime(x);
u3(1,:) = g_func(x);   v3(1,:) = 1i*g_prime(x);
% Laplacian (second-order central difference), Dirichlet at boundaries
e = ones(Nx,1);
Dxx = -spdiags([e -2*e e], -1:1, Nx, Nx) / dx^2;
Dxx(1,:) = 0; Dxx(end,:) = 0;
% --- Initialize memory B (discrete convolution with g') ---
s = 0:dt:10;  % integration variable s = x - t' in your approximation
ds = dt;
ws1 = exp(-gamma(1)* s);
ws2 = exp(-gamma(2)* s);
ws3 = exp(-gamma(3)* s);
B1 = zeros(1, Nx); B2 = B1; B3 = B1;
for j = 1:Nx
    xj = x(j);
    gprime_vals = g_prime( xj + s );  % vector
    B1(j) = sum(ws1 .* gprime_vals) * ds;
    B2(j) = sum(ws2 .* gprime_vals) * ds;
    B3(j) = sum(ws3 .* gprime_vals) * ds;
end
% --- One-step startup (forward Euler) ---
axt1 = 1i * V .* exp(-alpha(1)*(t(1)-t0)^2);
axt2 = 1i * V .* exp(-alpha(2)*(t(1)-t0)^2);
axt3 = 1i * V .* exp(-alpha(3)*(t(1)-t0)^2);
% Step from n=1 to n=2
u1(2,:) = u1(1,:) + 1i * dt * v1(1,:);
v1(2,:) = v1(1,:) + 1i * dt * ((Dxx * (u1(1,:).')).' + axt1 .* B1);
u2(2,:) = u2(1,:) + 1i * dt * v2(1,:);
v2(2,:) = v2(1,:) + 1i * dt * ((Dxx * (u2(1,:).')).' + axt2 .* B2);
u3(2,:) = u3(1,:) + 1i * dt * v3(1,:);
v3(2,:) = v3(1,:) + 1i * dt * ((Dxx * (u3(1,:).')).' + axt3 .* B3);
% --- Video writer and figure setup ---
filename = 'wave_multi_one.mp4';
writerObj = VideoWriter(filename,'MPEG-4');
writerObj.FrameRate = 40;
open(writerObj);
fig = figure;
fig.Position = [100 100 800 600];
tl = tiledlayout(fig, 1, 1);
ax1 = nexttile(tl, 1);
% --- Time stepping (leapfrog) ---
for n = 2:Nt-1
    tn = t(n);
    % update memory B[v] (exponential recurrence)
    B1 = exp(-gamma(1)*dt)*B1  + dt * v1(n,:);
    B2 = exp(-gamma(2)*dt)*B2  + dt * v2(n,:);
    B3 = exp(-gamma(3)*dt)*B3  + dt * v3(n,:);
    axt1 = 1i * V .* exp(-alpha(1)*(tn - t0)^2);
    axt2 = 1i * V .* exp(-alpha(2)*(tn - t0)^2);
    axt3 = 1i * V .* exp(-alpha(3)*(tn - t0)^2);
    % leapfrog updates
    u1(n+1,:) = u1(n-1,:) + 2i * dt * v1(n,:);
    v1(n+1,:) = v1(n-1,:) + 2i * dt * ((Dxx * (u1(n,:).')).' + axt1 .* B1);
    u2(n+1,:) = u2(n-1,:) + 2i * dt * v2(n,:);
    v2(n+1,:) = v2(n-1,:) + 2i * dt * ((Dxx * (u2(n,:).')).' + axt2 .* B2);
    u3(n+1,:) = u3(n-1,:) + 2i * dt * v3(n,:);
    v3(n+1,:) = v3(n-1,:) + 2i * dt * ((Dxx * (u3(n,:).')).' + axt3 .* B3);
    % normalize potential for plotting (avoid overwriting input V0)
    Vmax = max(abs(V));
    if Vmax == 0, W = V; else W = V / Vmax; end
    % Plot case 1
    axes(ax1); cla(ax1);
plot(ax1, x, real(u1(n,:)), 'c',x,real(u2(n,:)), ...
'r',x,real(u3(n,:)),'b', x, W, 'k', 'LineWidth', 2);
    xlabel(ax1, '$x$','Interpreter','latex','FontSize',20);
    axis(ax1,[min(x) max(x) -1.5 2]);
legend(ax1, ...
    sprintf('$\\alpha=%.1f,\\ \\gamma=%.1f$', alpha(1), gamma(1)), ...
    sprintf('$\\alpha=%.1f,\\ \\gamma=%.1f$', alpha(2), gamma(2)), ...
    sprintf('$\\alpha=%.1f,\\ \\gamma=%.1f$', alpha(3), gamma(3)), ...
    'Interpreter','latex','FontSize',20,'Location','southeast');
title(ax1, ...
      sprintf('${\\rm{Re}} u(x,t),\\ t=%.2f\\quad V(x)/\\max|V|,\\ \\max|V|=%.2f$', ...
            t(n), Vmax), ...
    'Interpreter','latex', ...
    'FontSize', 20);
    drawnow
    F = getframe(gcf);
    writeVideo(writerObj, F);
end
close(writerObj);
end
function V = Vnew(x, V0, x0)
% V(x) = V0(m) for x0(m) < x <= x0(m+1)
if nargin < 2
    V0 = [10 2 20];
end
if nargin < 3
    x0 = [-2 -1 1 2]/5;
end
% number of intervals should satisfy length(x0) = length(V0)+1
if length(x0) ~= length(V0)+1
    error('x0 must have length length(V0)+1');
end
V = zeros(size(x));
for j = 1:length(V0)
    idx = (x > x0(j)) & (x <= x0(j+1));
    V(idx) = V0(j);
end
end
\end{Verbatim}

\end{document}